\documentclass[12pt, twoside]{article}
\usepackage{amsmath,amsthm,amssymb}
\usepackage{times}
\usepackage{enumerate}
\usepackage{epsfig}
\usepackage{epstopdf}

\theoremstyle{definition}
\newtheorem{thm}{Theorem}[section]
\newtheorem{cor}[thm]{Corollary}
\newtheorem{lem}[thm]{Lemma}

\newtheorem{rem}[thm]{Remark}

\newtheorem{pro}[thm]{Proposition}

\numberwithin{equation}{section}

\newcommand{\subjclass}[1]{\bigskip\noindent\emph{2010 Mathematics Subject Classification:}\enspace#1}
\newcommand{\keywords}[1]{\noindent\emph{Keywords:}\enspace#1}

\def\div{\hbox{div}}

\def\div{\hbox{div}}

\def\O{\Omega}

\def\i{\infty}
\def\e{\epsilon}

\def\d{\delta}
\def\v{\varphi}
\def\n{\nabla}
\def\lt{L^2(\O)}

\def\hn{(H^1(\O))^m}

\def\si{\sum_{i=1}^{m}}
\def\sj{\sum_{j=1}^{m}}

\begin{document}


\baselineskip=17pt


\title{Existence result for degenerate cross-diffusion system with application to seawater intrusion}

\author{Jana Alkhayal\\
LaMA-Liban, lebanese university, P.O. Box 37 Tripoli, Lebanon.\\
jana.khayal@hotmail.com\\
Samar Issa\\
LaMA-Liban, lebanese university, P.O. Box 37 Tripoli, Lebanon.\\
samar\texttt{\char`_}issa@live.com\\
Mustapha Jazar\\
LaMA-Liban, lebanese university, P.O. Box 37 Tripoli, Lebanon.\\
mjazar@laser-lb.org\\
R\'egis Monneau\\
Universit\'e Paris-Est, CERMICS, Ecole des Ponts ParisTech,\\ 6 et 8 avenue Blaise Pascal, Cit\'e Descartes Champs-sur-Marne,\\ 77455 Marne-la-vall\'ee Cedex 2, France.\\
monneau@cermics.enpc.fr
}

\date{}

\maketitle


\begin{abstract}
In this paper, we study degenerate parabolic system, which is strongly coupled. We prove general existence result, but the uniqueness remains an open question. Our proof of existence is based on a crucial entropy estimate which both control the gradient of the solution and the non-negativity  of the solution. Our system are of porous medium type and our method applies to models in seawater intrusion.


\end{abstract}

\section{Introduction}

For the sake of simplicity, we will work on the torus $\Omega:=\mathbb T^N=\left(\mathbb R/\mathbb Z\right)^N$, with $N\geq 1$.\\
Let $\Omega_T:=(0,T)\times\O$ with $T>0$. Let an integer $m\geq 1$. Our purpose is to study a class of degenerate strongly coupled parabolic system of the form\\
\begin{equation}\label{sys0.1}
u_t^i= {\rm div}\left( u^i\overset{m}{\underset{j=1}{\sum}} A_{ij} \n u^j\right)\quad\quad\mbox{ in } \Omega_T, \quad\mbox{ for }i=1,\ldots, m.
\end{equation}
\noindent with the initial condition 
\begin{equation}\label{init1}
u^i(0,x)=u^i_0(x)\geq 0 \quad\quad\mbox{ a.e. in }\O,\quad\mbox{ for } i=1,\ldots,m.
\end{equation}
In the core of the paper we will assume that $A=(A_{ij})_{1\leq i,j\leq m}$ is a real $m\times m$ matrix (not necessarily symmetric) that satisfies the following positivity condition: we assume that there exists $\delta_0>0$, such that we have
\begin{equation}\label{eq+}
\xi^TA\, \xi \geq \delta_0|\xi|^2,\quad\mbox{ for all } \xi \in \mathbb R^m. 
\end{equation}
This condition can be weaken: see Subsection \ref{generalisation}.
Problem (\ref{sys0.1}) appears naturally in the modeling of seawater intrusion (see Subsection 1.2).
\subsection{Main results}
To introduce our main result, we need to define the nonnegative entropy function $\Psi$:\\
\begin{eqnarray}\label{Psi}
\Psi(a)-\frac{1}{e}=
\left\{\begin{array}{cl}
a\ln a&\mbox{ for}\quad a>0,\\
0&\mbox{ for}\quad a=0,\\
+\infty&\mbox{ for}\quad a<0,
\end{array}\right.
\end{eqnarray}
which is minimal for $\displaystyle a=\frac{1}{e}$.
\begin{thm}\textbf{(Existence for system (\ref{sys0.1}))}\label{th0.1}\\
Assume that $A$ satisfies (\ref{eq+}). For $i=1,\ldots,m$, let $u_0^i\ge 0$ in $\O$ satisfying 
\begin{equation}\label{psi0}
\sum_{i=1}^m\int_{\O}\Psi(u_0^i)<+ \infty,
\end{equation}
where $\Psi$ is given in (\ref{Psi}). Then there exists a function $u=(u^i)_{1\le i\le m} \in (L^2(0,T;H^1(\Omega))\ \cap\ C([0,T);(W^{1,\infty}(\O))'))^m$ solution in the sense of distributions of (\ref{sys0.1}),(\ref{init1}), with $u^i\geq 0$ a.e. in $\O_T$, for $i=1,\ldots,m$. Moreover this solution satisfies the following entropy estimate for a.e. $t_1, t_2\in (0,T)$, with $u^i(t_2)=u^i(t_2,\cdot)$:
\begin{equation}\label{entro}
\sum_{i=1}^m\int_{\O}\Psi(u^i(t_2))+\d_0\sum_{i=1}^m\int_{t_1}^{t_2}\int_{\O}|\nabla u^i|^2\le\sum_{i=1}^m\int_{\O}\Psi(u_0^i),
\end{equation}
where $\Psi$ is given in (\ref{Psi}).
\end{thm}
Here $\left\|A\right\|$ is the matrix norm defined as 
\begin{equation}\label{norme}
\left\|A\right\|=\sup_{\left|\xi\right|=1}\left|A\xi\right|.
\end{equation}
Notice that the entropy estimate (\ref{entro}) guarantees that $\n u^i \in L^2(0,T;L^2(\O))$, and therefore allows us to define the product $\displaystyle u^i\sum_{i=1}^mA_{ij}\n u^j$ in (\ref{sys0.1}). When our proofs were obtained, we realized that a similar entropy estimate has been obtained in \cite{CJ3} and \cite{CPZ} for a special system different from ours.
\begin{rem}(\textbf{Decreasing energy})\\
If $A$ is a symmetric matrix then a solution $u$ of system (\ref{sys0.1}) satisfies
$$\frac{d}{dt}\left(\si\sj\int_\O\frac{1}{2} A_{ij}u^iu^j\right)=-\si\int_\O u^i\left|\sj A_{ij}\n u^j\right|^2.$$
\end{rem}
\subsection{Application to seawater intrusion}
In this subsection, we describe briefly a model of seawater intrusion, which is particular case of our system (\ref{sys0.1}).\\
An aquifer is an underground layer of a porous and permeable rock through which water can move. On the one hand coastal aquifers contain freshwater and on the other hand saltwater from the sea can enter in the ground and replace the freshwater. We refer to \cite{BCSOH} for a general overview on seawater intrusion models.\\Now let $\nu=1-\e_0\in(0,1)$ where $$\e_0=\frac{\gamma_s-\gamma_f}{\gamma_s}$$ with $\gamma_s$ and $\gamma_f$ are the specific weight of the saltwater and freshwater respectively.\\


\begin{figure}[htbp]
\centerline{
\epsfig{figure=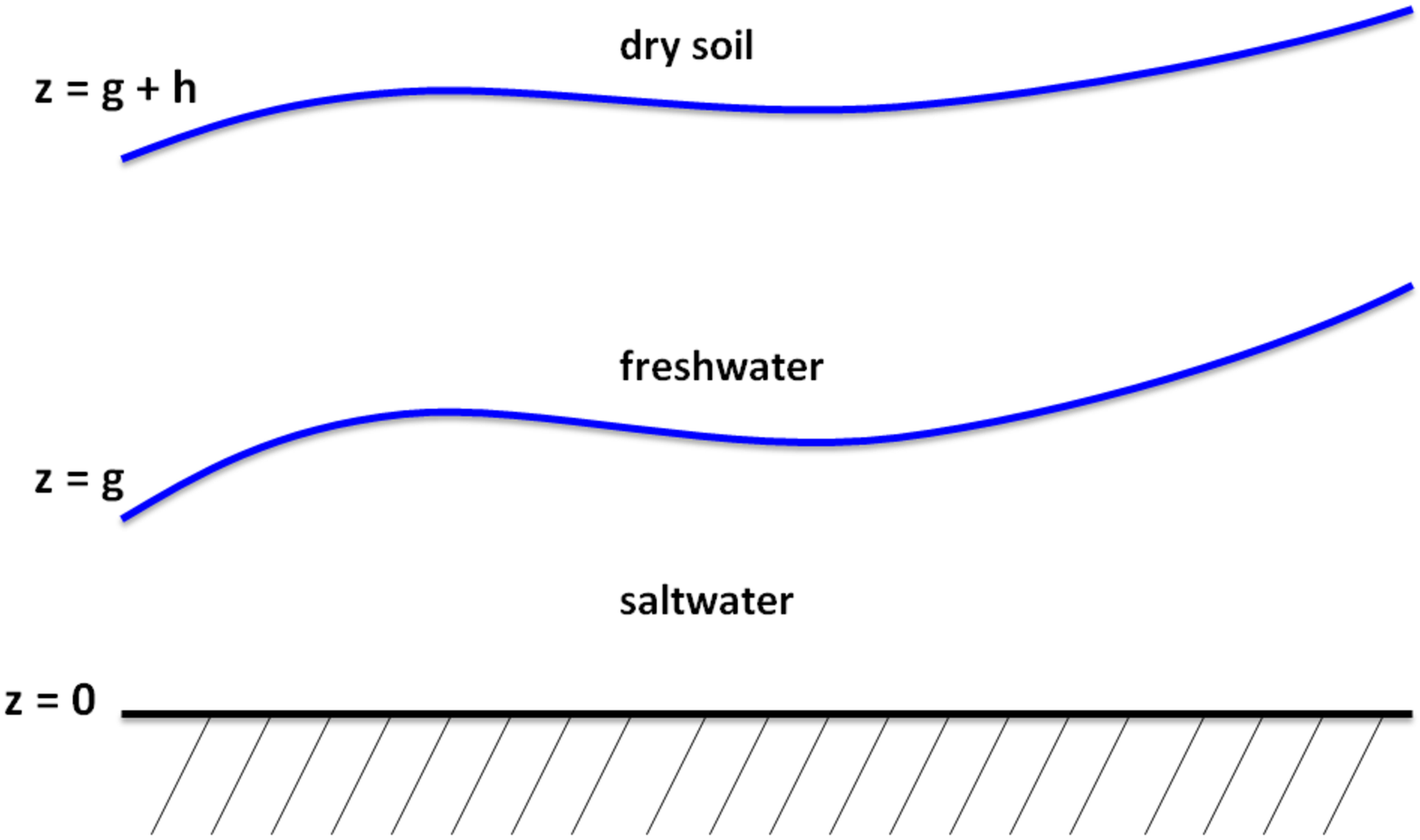, scale=0.4}
}
\label{fig:unconfined}\caption{Seawater intrusion in coastal aquifer}

\end{figure}

\noindent We assume that in the porous medium, the interface between the saltwater and the bedrock is given as $\{z=0\}$, the interface between the saltwater and the freshwater, which are assumed to be unmiscible, can be written as $\left\{z=g(t,x)\right\}$, and the interface between the freshwater and the dry soil can be written as $\left\{z=h(t,x)+g(t,x)\right\}$. Then the evolutions of $h$ and $g$ are given by a coupled nonlinear parabolic system (we refer to see \cite{RJ}) of the form
\begin{eqnarray}\label{uncon}
\left\{\begin{array}{llll}
h_t&=& {\rm div}\left\{ h\n(\nu( h+g)) \right\}  &\mbox{ in } \O_T,\\
g_t&=& {\rm div}\left\{g\n(\nu h+g) \right\} &\mbox{ in } \O_T,\\
\end{array}\right.
\end{eqnarray}
This is a particular case of (\ref{sys0.1}), where the $2\times 2$ matrix
\begin{equation}\label{A'}
A=
\begin{pmatrix} 
\nu & \nu \\
\nu & 1 
\end{pmatrix}
\end{equation}
satisfies (\ref{eq+}).
\subsection{Brief review of the litterature} 
The cross-diffusion systems, in particular the strongly coupled ones (for which the equations are coupled in the highest derivatives terms), are widely presented in different domains such as biology, chemistry, ecology, fluid mechanics and others. They are difficult to treat. Many of the standard results cannot be applied for such problems, such as the maximum principle. Hereafter, we cite several models where our method applies for most of them (see Section \ref{Gen} for more generalizations on our problem).

In \cite{SKT}, Shigesada, Kawasaki and Teramoto proposed a two-species SKT model in one-dimensional space which arises in population dynamics. It can be written in a generalized form with m-species as
\begin{equation}\label{population}
\displaystyle u_t^i-\Delta
\left[\left(\beta_i+\sum_{j=1}^{m}\alpha_{ij}u^j\right)u^i\right]
= \displaystyle\left(a_i-\sum_{j=1}^{m}b_{ij}u^j\right)u^i,\quad\mbox{ in }\O\times(0,T),
\end{equation}
where $u^i$, for $i=1,\ldots,m$, denotes the population density of the i-th species and $\beta_i$, $\alpha_{ij}$, $a_i$, $b_{ij}$ are nonnegative constants. The existence of a global solution for such problem in arbitrary space dimension is studied in \cite{WF}, where the quadratic form of the diffusion matrix is supposed positive definite. On the other hand, the two-species case was frequently studied, see for instance \cite{LNW,K,Y,GG,S} for dimensions $1$, $2$, and \cite{CJ3,PT,R,CJ2} for arbitrary dimension and appropriate conditions.

Another example of such problems is the electochemistry model studied by Choi, Huan and Lui in \cite{CHL} where they consider the general form
\begin{equation}\label{electro}
u^i_t=\sum_{\ell=1}^n \sj \frac{\partial}{\partial x_\ell}\left(a^{ij}_\ell(u)\frac{\partial u^j}{\partial x_\ell}\right),\quad u=(u^i)_{1\leq i\leq m}\quad\mbox{for}\quad i=1,\ldots,m, 
\end{equation}
and prove the existence of a weak solution of (\ref{electro}) under assumptions on the matrices $A_l(u)=(a_l^{ij}(u))_{1\leq i,j\leq m}$: it is continuous in $u$, its components are uniformly bounded with respect to $u$ and {its symmetric part} is definite positive. Their strategy of proof seeks to use Galerkin method to prove the existence of solutions to the linearized system and then to apply Schauder fixed-point theorem. Then they apply the results obtained to an electrochemistry model.

A third example of cross-diffusion models is the chemotaxis model introduced in \cite{LS}. The global existence for classical solutions of this model is studied by Hillen and Painter in \cite{HP} where they considered
\begin{eqnarray*}\label{chemotaxis1}
\left\{\begin{array}{clll}
u_t&=&\n\cdot(\n u- \chi(u,v)\n v),&\quad t>0, x\in\O\\
v_t&=&\mu \Delta v+ g(u,v),&\quad t>0, x\in\O,
\end{array}\right.
\end{eqnarray*}
on a $C^3$- differentiable compact Riemannian manifold without boundary, where the function $u$ describes the particle density, $v$ is the density of the external signal, the chemotactic cross-diffusion $\chi$ is assumed to be bounded, and the function $g$ describes production and degradation of the external stimulus. Another kind of chemotaxis model (the angiogenesis system) has been suggested and studied in \cite{CPZ}:
\begin{eqnarray*}\label{chemotaxis2}
\left\{\begin{array}{clll}
u_t&=& \kappa\Delta u-\nabla\cdot( u\,\chi(v) \n v),&\quad t>0, x\in\O\\
v_t&=&-v^m u,&\quad t>0, x\in\O,
\end{array}\right.
\end{eqnarray*}
where $m\geq 1$ and $\kappa$ is a constant.

Moreover, Alt and Luckhaus prove the existence in finite time of a solution for the following elliptic-parabolic problem
\begin{equation}\label{Alt-Luck}
\partial_t b^i(u)-\div (a^i(b(u),\nabla u))=f^i(b(u)),\quad\mbox{ in }\O\times(0,T),
\end{equation}
where $\O\subset R^N$ is open, bounded, and connected with Lipschitz boundary, $b$ is monotone and continuous gradient and $a$ is continuous and elliptic with some growth condition. This problem can be seen as a standard parabolic equation when $b(z)=z$.

Another problem is the Muskat Problem for Thin Fluid Layers of the form
\begin{eqnarray*}\label{Muskat}
\left\{\begin{array}{cll}
\partial_t f&=&(1+R)\partial_x(f\partial_x f)+R\partial_x(f\partial_x g),\\
\partial_t g&=&R_\mu\partial_x(g\partial_x f)+R_{\mu}\partial_x(g\partial_x g).\\
\end{array}\right.
\end{eqnarray*}
It models, \cite{EMM}, the motion of two fluids with different densities and viscosities in a porous meduim in one dimension, where $f$ and $g$ are the thickness of the two fluids and $R$, $R_{\mu}>0$ depending on the densities and the viscosities of the fluids. The authors in \cite{EMM,EM} studied the classical solutions of such problem. Moreover, weak solutions are established under different assumptions in \cite{ELM,M,LM1,LM2}.
\subsection{Strategy of the proof}\label{strategy}
In (\ref{sys0.1}), the elliptic part of the equation does not have a Lax-Milgram structure. Otherwise, our existence result is mainly based on the entropy estimate (\ref{entro}). It is difficult to get this entropy estimate directly (we do not have enough regularity to do it), so we proceed by approximations.\\
\textbf{Approximation 1:}\\
We discretize in time system (\ref{sys0.1}), with a time step $\Delta t=T/K$, where $K\in\mathbb N^*$. Then for a given $u^{n}=(u^{i,n})_{1\leq i\leq m}\in (H^1(\O))^m$, we consider the implicit scheme which is an elliptic system:
\begin{equation}\label{sys2'}
\displaystyle{\frac{u^{i,n+1}-u^{i,n}}{\Delta t}} = {\rm div} \left\{u^{i,n+1}\sum_{j=1}^m A_{ij}\nabla u^{j,n+1}\right\}.
\end{equation}
\textbf{Approximation 2:}\\
We regularize the right-hand term of (\ref{sys2'}).
To do that, we take $\eta>0$ and $0<\e<1<\ell$, and we choose the following regularization 
\begin{equation}\label{reg}
\displaystyle{\frac{u^{i,n+1}-u^{i,n}}{\Delta t}} = {\rm div}\left\{T^{\e,\ell}(u^{i,n+1})\sum_{j=1}^m A_{ij}\n \rho_\eta\star\rho_\eta\star u^{j,n+1}\right\},
\end{equation}
where $T^{\e,\ell}$ is truncation operator defined as
\begin{eqnarray}
T^{\e,\ell}(a):=
\left\{\begin{array}{cll}\label{T}
\e &\mbox{ if }a\leq \e,\\
a &\mbox{ if }\e\leq a\leq \ell,\\
\ell &\mbox{ if }a\geq \ell,
\end{array}\right.
\end{eqnarray}
and the mollifier $\rho_\eta(x)=\eta^{-N}\rho\left(x/\eta\right)$ with $\rho \in C^\infty_c(\mathbb R^N)$, $\rho\geq 0$, $\int_{\mathbb R^N} \rho=1$ and $\rho(-x)=\rho(x)$.\\
Now, with the convolution by $\rho_\eta$ in (\ref{reg}), the term $\n \rho_\eta\star\rho_\eta\star u^{j,n+1}$ behaves like $u^{j,n+1}$. \\
Note that, considering the $\mathbb Z^N$- periodic extension on $\mathbb R^N$ of $u^{j,n+1}$, the convolution $\rho_\eta\star u^{j,n+1}$ is possible over $\mathbb R^N$.\\
\textbf{Approximation 3:}\\
Let $\delta >0$. We will add a second order term like $\delta\Delta u^i$ to equation (\ref{reg}) in order to obtain an elliptic one. More specifically, we consider ${\rm div}\left(\delta T^{\e,\ell}(u^i) \n u^i\right)$ instead of $\delta \Delta u^i$, to  keep an entropy estimate.\\
Then we freeze the coefficients $u^{i,n+1}$ on the right-hand side to make a linear structure (these coefficients are now called $\d T^{\e,\ell}(v^{i,n+1})$), we obtain the following modified system:
\begin{equation}\label{sys1}
\displaystyle{\frac{u^{i,n+1}-u^{i,n}}{\Delta t}}={\rm div}\left\{T^{\epsilon,\ell}(v^{i,n+1})\left(\overset{m}{\underset{j=1}{\sum}} A_{ij}\nabla\rho_\eta\star\rho_\eta\star u^{j,n+1}+\delta\nabla u^{i,n+1}\right)\right\}.
\end{equation}
We will look for fixed points solutions $v^{i,n+1}=u^{i,n+1}$ of this modified system. Finally, we will recover the expected result dropping one after one all the approximations.
\subsection{Organization of the paper}
In Section 2, we recall some useful tools. In Section 3, we study system (\ref{sys0.1}). By discretizing our problem on time, in Subsection 3.1, we obtain an elliptic problem. We use the Lax-Milgram theorem to show the existence of a unique solution to the linear problem (\ref{sys1}). We demonstrate, in Subsection 3.2, the existence of a solution of the nonlinear problem, using the Schaefer's fixed point theorem.\\Then we pass to the limit in the following order:
$(\Delta t,\e)\rightarrow (0,0)$ in Subsection 3.3, $(\ell,\eta)\rightarrow (\infty,0)$ in Subsection 3.4 and $\delta\rightarrow 0$ in Subsection 3.5.
\noindent Generalizations (including more general matrices $A$ or tensors) will be presented in Section 4. We end with an Appendix showing some technical results in Section 5.

\section{Preliminary tools }
\begin{thm}\label{schaefer}{\bf (Schaefer's fixed point theorem)}\cite[Theorem 4 page 504]{Evans}\\
Let $X$ be a real Banach space. Suppose that
$$\Phi : X\to X$$
is a continuous and compact mapping. Assume further that the set
$$\left\{u\in X,\quad u=\lambda \Phi(u) \quad \mbox{for some}\quad \lambda
  \in [0,1] \right\}$$
is bounded. Then $\Phi$ has a fixed point.
\end{thm}
\begin{pro}\label{compacite}{\bf (Aubin's lemma)}\cite{Simon}\\
For any $T>0$, and $\O=\mathbb T^N$, let $E$ denote the space $$E:=\left\{g\in L^2((0,T);H^1(\O)) \mbox{ and }
g_t\in L^2((0,T);H^{-1}(\O)) \right\},$$
endowed with the Hilbert norm
$$\left\|\omega\right\|_E=\left(\left\|\omega\right\|^2_{L^2(0,T;H^1(\O))}+\left\|\omega_t\right\|^2_{L^2(0,T;H^{-1}(\O))}\right)^\frac{1}{2}.$$  The embedding
$$ E\hookrightarrow L^2((0,T);L^2(\O)) \quad \mbox{is compact.}$$
\end{pro}
On the other hand, it follows from \cite[{P}roposition 2.1 and {T}heorem 3.1, {C}hapter 1]{cont} that the embedding $$ E\hookrightarrow C([0,T];L^2(\O))\quad \mbox{is continuous}.$$
\begin{lem}\textbf{(Simon's Lemma)}\label{lemsimon}\cite{Simon}\\
Let $X$, $B$ and $Y$ three Banach spaces, where $X\hookrightarrow B$
with compact embedding and $B\hookrightarrow Y$ with continuous
embedding. If $(g^n)_n$ is a sequence such that
$$
\|g^n\|_{L^q(0,T;B)}+ \|g^n\|_{L^1(0,T;X)}+\|g_t^n\|_{L^1(0,T;Y)}\le
C,
$$
where $1<q\le \infty$, and $C$ is a constant independent of $n$,
then $(g^n)_n$ is relatively compact in $L^p(0,T;B)$ for all $1\le
p<q$.
\end{lem}
Now we will present the variant of  the original result of Simon's lemma \cite[Corollary 6, page 87]{Simon}. First of all, let us define the norm $\left\|.\right\|_{\mbox{Var}([a,b);Y)}$ where $Y$ is a Banach space with the norm $\left\|.\right\|_Y$.\\
For a function $g:[a,b)\to Y$, we set
\begin{equation}\label{Var}
 \left\|g\right\|_{\mbox{Var}([a,b);Y)}=\sup \sum_j \left\|g(a_{j+1})-g(a_j)\right\|_Y
\end{equation}
over all possible finite partitions:
$$a\le a_0<\dots <a_k<b$$
\begin{thm}\textbf{(Variant of Simon's Lemma)}\label{lemsimonvar}\\
Let $X$, $B$ and $Y$ three Banach spaces, where $X\hookrightarrow B$
with compact embedding and $B\hookrightarrow Y$ with continuous
embedding. Let $(g^n)_n$ be a sequence such that
\begin{equation}\label{eq::jr1}
\left\|g^n\right\|_{L^1(0,T;X)}+\left\|g^n\right\|_{L^q(0,T;B)}+\left\|g^n\right\|_{\mbox{Var}([0,T);Y)}\le C,
\end{equation}
where $1<q< \infty$, and $C$ is a constant independent of $n$.
Then $(g^n)_n$ is relatively compact in $L^p(0,T;B)$ for all $1\le
p<q$.
\end{thm}
\begin{proof}
\noindent {\bf Step 1: Regularization of the sequence}\\
Let $\bar \rho \in C^\infty_c(\mathbb R)$ with $\bar \rho\ge 0$, $\int_{\mathbb R} \bar \rho =1$ and $\mbox{supp}\ \bar \rho \subset (-1,1)$.
For $\varepsilon>0$, we set
$$\bar \rho_\varepsilon (x)= \varepsilon^{-1}\bar \rho(\varepsilon^{-1}x).$$
We extend $g^n=g^n(t)$ by zero outside the time interval $[0,T)$.
Because $q<+\infty$, we see that for each $n$, we choose some $0< \varepsilon_n\to 0$ as $n\to +\infty$ such that
\begin{equation}\label{eq::jr2}
\left\|\bar g^n-g^n\right\|_{L^q(0,T;B)}\rightarrow 0 \quad \mbox{as}\quad n\to +\infty,\quad \mbox{with}\quad \bar g^n = \bar \rho_{\varepsilon_n} \star g^n
\end{equation}
For any $\delta>0$ small enough, we also have for $n$ large enough (such that $\varepsilon_n<\delta$):
$$\left\|\bar g^n\right\|_{L^1(\delta,T-\delta;X)}\le \left\|g^n\right\|_{L^1(0,T;X)} \le C$$
and
\begin{equation}\label{eq::jr3}
\left\|\bar g_t^n\right\|_{L^1(\delta,T-\delta;Y)} \le \left\|g^n\right\|_{\mbox{Var}([0,T);Y)}\le C
\end{equation}
\noindent {\bf Step 2: Checking (\ref{eq::jr3})}\\
By (\ref{eq::jr1}) there exists a sequence of  step functions $f_\eta$ 
which approximates uniformly $g^n$ on $[0,T)$ as $\eta\to 0$, with moreover satisfies
$$\left\|f_\eta\right\|_{\mbox{Var}([0,T);Y)} \to \left\| g^n\right\|_{\mbox{Var}([0,T);Y)}.$$
Therefore we get easily (for $\varepsilon_n< \delta$)
$$\left\|(\bar  \rho_{\varepsilon_n} \star f_\eta )_t\right\|_{L^1(\delta,T-\delta;Y)} \le \left\|f_\eta\right\|_{\mbox{Var}([0,T);Y)}$$
which implies (\ref{eq::jr3}), when we pass to the limit as $\eta$ goes to zero.\\
\noindent {\bf Step 3: Conclusion}\\
We can then apply Corollary 6 in \cite{Simon} to deduce that $\bar g^n$ is relatively compact in $L^p(0,T;B)$ for all $1\le
p<q$. Because of (\ref{eq::jr2}), we deduce that this is also the case for the sequence $(g^n)_n$, which ends the proof of the Theorem.
$\qedhere$
\end{proof}
\section{Existence for system (\ref{sys0.1})}
Our goal is to prove Theorem \ref{th0.1} in order to get the existence of a solution for system (\ref{sys0.1}).
\subsection{Existence for the linear elliptic problem (\ref{sys1})}
In this subsection we prove the existence, via Lax-Milgram theorem, of the unique solution for the linear elliptic system (\ref{sys1}).\\
Let us recall our linear elliptic system. Assume that $A$ is any $m\times m$ real matrix. Let $v^{n+1}=(v^{i,n+1})_{1\leq i\leq m}\in (L^2(\O))^m$ and $u^n=(u^{i,n})_{1\leq i\leq m}\in (H^1(\O))^m$. Then for all $\Delta t$, $\e$, $\ell$, $\eta$, $\delta>0$, with $\e<1<\ell$ and $\Delta t<\tau$ where $\tau$ is given in (\ref{eq4}), we look for the solution $u^{n+1}=(u^{i,n+1})_{1\leq i\leq m}$ of the following system:\\
\begin{eqnarray}\label{sys2}
\left\{\begin{array}{cll}
\displaystyle{\frac{u^{i,n+1}-u^{i,n}}{\Delta t}} &=& {\rm div} \left\{J_{\epsilon,\ell,\eta,\delta}^{i}(v^{n+1},u^{n+1})\right\}\quad\mbox{in}\quad {\mathcal D}'(\O),\\\\
J_{\epsilon,\ell,\eta,\delta}^{i}(v^{n+1},u^{n+1})&=&T^{\epsilon,\ell}(v^{i,n+1})\left\{\overset{m}{\underset{j=1}{\sum}} A_{ij}\nabla\rho_\eta\star\rho_\eta\star u^{j,n+1}+\delta \nabla u^{i,n+1}\right\},
\end{array}\right.
\end{eqnarray}
where $T^{\e,\ell}$ is given in (\ref{T}).\\
\begin{pro}\textbf{(Existence for system (\ref{sys2}))}\label{theo1}\\
 Assume that $A$ is any $m\times m$ real matrix. Let $\Delta t$, $\e$, $\ell$, $\eta$, $\delta>0$, with $\e<1<\ell$, such that
\begin{equation}\label{eq4}
\Delta t< \frac{\delta\e\eta^2}{C_0{^2}\ell^2 \left\|A\right\|^2}:=\tau,
\end{equation}
where 
\begin{equation}\label{C_0}
{C_0=\left\|\n \rho\right\|_{L^1(\mathbb R^N)}}.
\end{equation}
Then for $n\in \mathbb N$, for a given $v^{n+1}=(v^{i,n+1})_{1\leq i\leq m}\in (L^2(\Omega))^m$ and $u^{n}=(u^{i,n})_{1\leq i\leq m}\in
(H^1(\Omega))^m$, there exists a unique function $u^{n+1}=(u^{i,n+1})_{1\leq i\leq m}\in(H^1(\O))^m$ solution of system {(\ref{sys2})}. Moreover, this solution $u^{n+1}$ satisfies the following estimate
 \begin{equation}\label{est0}
\left(1-\frac{\Delta t }{\tau}\right)\left\|u^{n+1}\right\|^2_{(L^2(\O))^m}+\Delta t \e \delta\left\|\n u^{n+1}\right\|^2_{(L^2(\O))^m}\leq \left\|u^{n}\right\|^2_{(L^2(\O))^m},
\end{equation}
where $\tau$ is given in (\ref{eq4}).
\end{pro}
\begin{proof}
The proof is done in four steps using Lax-Milgram theorem.\\ First {of all}, let us define for all $u^{n+1}=(u^{i,n+1})_{1\leq i\leq m}$ and $\varphi=(\varphi^i)_{1\leq i\leq m}\in (H^1(\Omega))^m$, the following bilinear form:
\begin{eqnarray*} 
a(u^{n+1},\varphi)&=&\si\int_\O u^{i,n+1}\varphi^i + \Delta t\sum_{i,j=1}^m \int_{\Omega}T^{\e,\ell}(v^{i,n+1})A_{ij}
\left(\nabla\rho_\eta\star\rho_\eta\star u^{j,n+1}\right)\cdot\nabla\varphi^i\\ 
& &+ \Delta t\delta \si\int_{\O} T^{\e,\ell}(v^{i,n+1})\nabla u^{i,n+1}\cdot\nabla\varphi^i,
\end{eqnarray*}
which can be also {rewritten} as 
\begin{eqnarray*} 
a(u^{n+1},\varphi)&=&\left\langle u^{n+1},\varphi\right\rangle_{(L^2(\O))^m}+ \Delta t\left\langle T^{\e,\ell}(v^{n+1})\nabla\varphi, A
\nabla\rho_\eta\star\rho_\eta\star u^{n+1}\right\rangle_{(L^2(\O))^m}\\ 
& &+ \Delta t\delta \left\langle T^{\e,\ell}(v^{n+1})\nabla\varphi,\nabla u^{n+1}\right\rangle_{(L^2(\O))^m},\\
\end{eqnarray*}
where $\left\langle\cdot , \cdot\right\rangle_{(L^2(\O))^m}$ denotes the scalar product on $(L^2(\O))^m$, and the following linear form:
$$L(\varphi)=\sum_{i=1}^m\int_{\Omega}u^{i,n}\varphi^i=\left\langle u^{n},\varphi\right\rangle_{(L^2(\O))^m}.$$
\textbf{Step 1: Continuity of $a$} \\
For every $n\in\mathbb{N}$, $u^{n+1}$ and $\varphi\in(H^1(\Omega))^m$, we have
\begin{eqnarray*}
 |a(u^{n+1},\varphi)| &\le & \|
u^{n+1}\|_{(L^2(\Omega)^m}\|\varphi\|_{(L^2(\Omega))^m} +\Delta t \ell \|A\|\|\nabla\rho_\eta\star\rho_\eta\star u^{n+1}\|_{(L^2(\Omega))^m}\|\nabla\varphi\|_{(L^2(\Omega))^m}\\
& & + \Delta t
 \delta \ell \|\nabla u^{n+1}\|_{(L^2(\Omega))^m}\|\nabla\varphi\|_{(L^2(\Omega))^m}\\
 &\le & 3 \max(1,\Delta t \ell \|A\|,\Delta t \delta
 \ell)\|u^{n+1}\|_{(H^1(\Omega))^m}\|\varphi\|_{(H^1(\Omega))^m}.
\end{eqnarray*}
where $\left\|A\right\|$ is given in (\ref{norme}) and we have used the fact that 
\begin{equation}\label{conv}
\left\|\n \rho_\eta\star\rho_\eta\star u^{n+1}\right\|_{(L^2(\O))^m}\leq \left\|\n u^{n+1}\right\|_{(L^2(\O))^m},
\end{equation}
and
\begin{equation}\label{T'}
\e\leq T^{\e,\ell}(a) \leq \ell,\quad \mbox{ for all } a\in \mathbb R.
\end{equation}
\textbf{Step 2: Coercivity of $a$}\\
For all $\varphi\in(H^1(\Omega))^m$, we have that $a(\varphi,\varphi)=a_0(\varphi,\varphi)+a_1(\varphi,\varphi)$, where
 $$a_0(\varphi,\varphi)= \left\|\v\right\|^2_{(L^2(\O))^m}+\Delta t\d \left\langle T^{\e,\ell}(\v)\n\v,\n\v\right\rangle_{(L^2(\O))^m}$$
and $$a_1(\varphi,\varphi)=\Delta t\left\langle T^{\e,\ell}(\v)\n\v,A \n\rho_\eta\star\rho_\eta\star\v\right\rangle_{(L^2(\O))^m}.$$
On the one hand, we already have the coercivity of $a_0 $:
\begin{eqnarray*}
a_0(\varphi,\varphi)&\geq&\|\v\|_{(L^2(\Omega))^m}^2 + \Delta t\d\e \|\n\v\|_{(L^2(\Omega))^m}^2.
\end{eqnarray*}
On the other hand, we have
\begin{eqnarray*}
\left|a_1(\varphi,\varphi)\right|&\leq& \Delta t \ell \left\|A\right\|\left\|\n\rho_\eta\star\rho_\eta\star\v\right\|_{(L^2(\O))^m}\left\|\n\v\right\|_{(L^2(\O))^m}\\
&\leq& \Delta t \ell \left\|A\right\|\left(\frac{1}{2\alpha}\left\|\n\rho_\eta\star\rho_\eta\star\v\right\|^2_{(L^2(\O))^m}+\frac{\alpha}{2}\left\|\n\v\right\|^2_{(L^2(\O))^m}\right)\\
&\leq& \frac{\Delta t \ell^2\left\|A\right\|^2{C_0^2}}{2\delta \e\eta^2}\left\|\varphi\right\|^2_{(L^2(\O))^m}+\frac{\Delta t \e \delta}{2}\left\|\n\varphi\right\|^2_{(L^2(\O))^m},
\end{eqnarray*}
where in the second line we have used Young's inequality, and chosen $\displaystyle \alpha=\frac{\delta\e}{\left\|A\right\|\ell}$ in the third line, with ${C_0}$ is given in (\ref{C_0}) and $\left\|A\right\|$ is given in (\ref{norme}). So we get that
\begin{eqnarray}\label{a}
a(\varphi,\varphi)&\geq& \left(1-\frac{\Delta t }{2\tau}\right)\left\|\varphi\right\|^2_{(L^2(\O))^m}+\frac{\Delta t \e \delta}{2}\left\|\n\varphi\right\|^2_{(L^2(\O))^m}
\end{eqnarray}
is coercive, since $\Delta t<\tau$ where $\tau$ is given in (\ref{eq4}).\\\\
\textbf{Step 3: Existence by Lax-Milgram}\\
\noindent It is clear that $L$ is linear and continuous on
$\hn$. Then by Step 1, Step 2 and Lax-Milgram theorem there exists a unique solution, $u^{n+1}$,
of system (\ref{sys2}).\\\\
\noindent{\textbf{Step 4: Proof of estimate (\ref{est0})}}\\
Using (\ref{a}) and the fact that $a(u^{n+1},u^{n+1})=L(u^{n+1})$ we get
\begin{eqnarray*}
\left(1-\frac{\Delta t }{2\tau}\right)\left\|u^{n+1}\right\|^2_{(L^2(\O))^m}+\frac{\Delta t \e \delta}{2}\left\|\n u^{n+1}\right\|^2_{(L^2(\O))^m}&\leq& \left\langle u^{i,n},u^{i,n+1}\right\rangle_{(L^2(\O))^m}\\
&\leq&\frac{1}{2} \left\|u^{n}\right\|^2_{(L^2(\O))^m}+\frac{1}{2} \left\|u^{n+1}\right\|^2_{(L^2(\O))^m},
\end{eqnarray*}
which gives us the estimate (\ref{est0}).
$\qedhere$
\end{proof}
\subsection{Existence for the nonlinear time-discrete problem}
In this subsection we prove the existence, using Schaefer's fixed point theorem, of a solution for the nonlinear time discrete-system {(\ref{sys3})} given below. Moreover, {we also show that} this solution satisfies a suitable entropy estimate.\\
First, {to} {present} our result we need to choose a function $\Psi_{\e,\ell}$ which is continuous, convex and satisfies that $\displaystyle\Psi''_{\e,\ell}(x)=\frac{1}{T^{\e,\ell}(x)}$, where ${T^{\e,\ell}}$ is given in (\ref{T}). So let
\begin{eqnarray}
\Psi_{\e,\ell}({a}){-\frac{1}{e}}=
\left\{\begin{array}{lll}\label{psi2}
\frac{{a}^2}{2\e}+{a}\ln\epsilon-\frac{\e}{2} &\quad\mbox{ if } {a}\leq\e,&\\\\
{a}\ln {a} &\quad\mbox{ if } \e<{a}\leq\ell,& \\\\
\frac{{a}^2}{2\ell}+{a}\ln\ell-\frac{\ell}{2} &\quad\mbox{ if } {a}>{\ell.}& 
\end{array}\right.
\end{eqnarray}
Let us introduce our nonlinear time discrete system: Assume that $A$ satisfies (\ref{eq+}). {Let} ${u^{0}}=(u^{i,0})_{1\leq i\leq m}:=u_0=(u^i_0)_{1\leq i\leq m}$ that satisfies
\begin{equation}\label{C_1}
C_1:=\si\int_\O \Psi_{\e,\ell}(u^i_0)<+\infty,
\end{equation}
such that $u^i_0\geq 0$ in $\O$ for $i=1,\ldots,m$. Then for all $\Delta t$, $\e$, $\ell$, $\eta$, $\delta>0$, with $\e<1<\ell$ and $\Delta t<\tau$ where $\tau$ is given in (\ref{eq4}), for $n\in\mathbb N$, we look for a solution $u^{n+1}=(u^{i,n+1})_{1\leq i\leq m}$ of the following system:
\begin{eqnarray}\label{sys3}
\left\{\begin{array}{clll}
\displaystyle{\frac{u^{i,n+1}-u^{i,n}}{\Delta t}}&=&{\rm div} \left\{J_{\e,\ell,\eta,\d}^{i}(u^{n+1},u^{n+1})\right\}&\quad\mbox{ in } \mathcal D'(\Omega), {\mbox{ for } n\geq 0}\\\\
{u^{i,0}(x)}&=&{u^i_0(x)}&{\quad\mbox{ in } \Omega,}
\end{array}\right.
\end{eqnarray}
where $J_{\e,\ell,\eta,\d}^{i}$ is given in system {(\ref{sys2})}, and $T^{\e,\ell}$ is given in (\ref{T}).
\begin{pro}\label{th2}\textbf{(Existence for system (\ref{sys3}))}\\
Assume that $A$ satisfies (\ref{eq+}). {Let} ${u_0}=(u^i_0)_{1\leq i\leq m}$ that satisfies (\ref{C_1}), such that $u^i_0\geq 0$ {a.e.} in $\O$ for $i=1,\ldots,m$. Then for all $\Delta t$, $\e$, $\ell$, $\eta$, $\delta>0$, with $\e<1<\ell$ and $\Delta t<\tau$ where $\tau$ is given in {(\ref{eq4}),} there exists a sequence of functions $u^{n+1}=(u^{i,n+1})_{1\leq i\leq m}\in  \left(H^1(\O)\right)^m$ {for $n\in\mathbb N$}, solution of system (\ref{sys3}), that satisfies the following entropy estimate:\\ 
\begin{equation}\label{2n}
\sum_{i=1}^m\int_{\O}{\Psi}_{\e,\ell}(u^{i,n+1})+\d\Delta t\sum_{i=1}^m\sum_{k=0}^{n}\int_{\O}|\n u^{i,k+1}|^2+\delta_0\Delta t\sum_{i=1}^m\sum_{k=0}^{n}\int_{\O}|\n \rho_{\eta}\ {\star}\ u^{i,k+1}|^2\le\sum_{i=1}^m \int_{\O} {\Psi}_{\e,\ell}(u^{i}_0),
\end{equation}
where ${\Psi}_{\e,\ell}$ is given in {(\ref{psi2})}.
\end{pro}
\begin{proof}
Our proof is based on the Schaefer's fixed point theorem. So we need to define, for a given $w:=u^n=(u^{i,n})_{1\leq i\leq m}\in (L^2(\O))^m$ and $v:=v^{n+1}=(v^{i,n+1})_{1\leq i\leq m}\in (L^2(\O))^m$, the map $\Phi$ as:
$$\begin{array}{ccccc}
\Phi & : &(L^2(\O))^m  & \to & (L^2(\O))^m \\
 & & v & \mapsto & u \\
\end{array}$$
where $u:=u^{n+1}=(u^{i,n+1})_{1\leq i\leq m}=\Phi(v^{n+1})\in (H^1(\O))^m$ is the unique solution of system (\ref{sys2}), given by Proposition \ref{theo1}.\\\\
\noindent\textbf{Step 1: Continuity of $\Phi$}\\
Let us consider the sequence $ v_k$ such that
$$\left\{\begin{array}{l}
v_k\in (L^2(\O))^m,\\
\\
v_k \longrightarrow v \quad \mbox{in}\quad (L^2(\O))^m.
\end{array}\right.$$ 
We want to prove that the sequence $u_k=\Phi(v_k)\longrightarrow u=\Phi(v)$ to get the continuity of $\Phi$.
From the estimate (\ref{est0}), we deduce that $u_k$ is
bounded in $(H^1(\O))^m$. Therefore, up to a subsequence, we have
$$
\left\{
\begin{array}{ll}
u_k \rightharpoonup u & \mbox{weakly in }  \; (H^1(\O))^m,\\
\mbox{and} &\\
u_k \rightarrow u & \mbox{strongly in }  \;(L^2(\O))^m,\\
\end{array}
\right.$$
where the strong convergence arises because $\O$ is compact. Thus, by the definition of the truncation operator $T^{\e,\ell}$, we can see that $T^{\e,\ell}$ is continuous and bounded, then by dominated convergence theorem, we have that 
$$T^{\e,\ell}(v_{k}^i) \longrightarrow T^{\e,\ell}(v^i) \quad \mbox{in } L^2(\O),\quad\mbox{ for }i=1,\ldots, m.$$
Now we have 
\begin{equation}\label{sys}
\displaystyle{\frac{u_k^{i}-w^i}{\Delta t}} =
{\rm div} \left\{J_{\e,\ell,\eta,\d}^{i}(v_k,u_k)\right\} \quad\mbox{ in }\, \mathcal D'(\O).
\end{equation}
This system also holds in $H^{-1}(\O)$, because $J_{\e,\ell,\eta,\d}^{i}(v_k,u_k)\in L^2(\O)$. Hence by multiplying this system by a test function in $\hn$ and integrating over $\O$ for the bracket $\left\langle \cdot,\cdot\right\rangle_{H^{-1}(\O)\times H^1(\O)}$, we can pass
directly to the limit in (\ref{sys}) as $k$ tends to $+\i$, and we get 
\begin{eqnarray}
\displaystyle{\frac{u^{i}-w^{i}}{\Delta t}} =
{\rm div} \{J_{\e,\ell,\eta,\d}^{i}(v,u)\} \quad\mbox{ in }\, \mathcal D'(\O).
\end{eqnarray}
where we used in particular the weak $L^2$ - strong $L^2$ convergence in the product $T^{\e,\ell}(v_k)\n u_k$. Then $u=(u^i)_{1\leq i\leq m}=\Phi(v)$ is a
solution {of} system {(\ref{sys2})}. Finally, by uniqueness of the
solutions of (\ref{sys2}), we deduce that the limit $u$ does not
depend on the choice of the subsequence, and then that the full
sequence converges:
$$u_k \rightarrow u \quad \mbox{strongly in}\quad (L^2(\O))^m, \quad
\mbox{with}\quad u =\Phi(v).$$\\
\noindent {\bf Step 2: Compactness of $\Phi$ }\\
By the definition of $\Phi$ we can see that for a bounded sequence $(v_k)_k$ in $(L^2(\O))^m$, $\Phi (v_k)=u_k$ converges strongly in $(L^2(\O))^m$ up to a subsequence, which implies the compactness of $\Phi$.\\\\
\noindent {\bf Step 3: A priori bounds on the solutions of $v=\lambda\Phi(v)$}\\
Let us consider a solution $v$ of
\begin{equation*}
 v=\lambda \Phi(v) \quad \mbox{for some}\quad \lambda\in [0,1].
\end{equation*}
By (\ref{est0}) we see that there exists a constant $C_2=C_2(\Delta t,\e,...)$ such that for any given $w\in (L^2(\O))^m$, we have $\left\|\Phi(v)\right\|_{(H^1(\O))^m}\leq C_2 \left\|w\right\|_{(L^2(\O))^m}$.
Hence $v=\lambda \Phi(v)$ is bounded.\\\\
\noindent {\bf Step 4: Existence of a solution}\\
 Now, we can apply Schaefer's
fixed point Theorem (Theorem \ref{schaefer}), to deduce that $\Phi$
has a fixed point $u^{n+1}$ on $(L^2(\O))^m$. This implies the existence of a solution $u^{n+1}$ of system (\ref{sys3}).\\\\
\noindent {\bf Step 5: Proof of estimate (\ref{2n})}\\
We have,
\begin{eqnarray*}
&&\sum_{i=1}^m\int_{\O}\frac{\Psi_{\e,\ell}(u^{i,n+1})-\Psi_{\e,\ell}(u^{i,n})}{\Delta t}\\
&\le &\sum_{i=1}^m\int_{\O} \left(\frac{u^{i,n+1}-u^{i,n}}{\Delta t}\right)\Psi_{\e,\ell}'(u^{i,n+1})\\
&=&\sum_{i=1}^m\left\langle \frac{u^{i,n+1}-u^{i,n}}{\Delta t},\Psi_{\e,\ell}'(u^{i,n+1})\right\rangle_{H^{-1}(\O)\times H^1(\O)}\\
&= & -\sum_{i=1}^m\left\langle \d T^{\e,\ell}(u^{i,n+1})\n u^{i,n+1}+T^{\e,\ell}(u^{i,n+1})\sum_{j=1}^mA_{ij}\n\rho_{\eta}\star\rho_{\eta}\star u^{j,n+1},\Psi_{\e,\ell}''(u^{i,n+1})\n u^{i,n+1}\right\rangle_{L^2(\O)}\\
&= & -\sum_{i=1}^m\left\{\d\int_{\O} |\n u^{i,n+1}|^2 + \int_{\O}\sum_{j=1}^m\n\rho_{\eta}\star u^{i,n+1} A_{ij}\n\rho_{\eta}\star u^{j,n+1}\right\}\\
&\le & -\sum_{i=1}^m\d\int_{\O} |\n u^{i,n+1}|^2 -\delta_0\sum_{i=1}^m\int_{\O}|\n\rho_{\eta}\star u^{i,n+1}|^2,\\
\end{eqnarray*}
where we have used, in the second line, the convexity inequality on $\Psi_{\e,\ell}$. In the third line, we used the fact that $\displaystyle\frac{u^{i,n+1}-u^{i,n}}{\Delta t}\in H^{-1}(\O)$ and that $\n\Psi'_{\e,\ell}(u^{i,n+1})=\Psi''_{\e,\ell}(u^{i,n+1})\nabla u^{i,n+1}\in L^2(\O)$ because $\Psi'_{\e,\ell}(u^{i,n+1})\in C^1(\mathbb R)$, see \cite[Proposition IX.5, page 155]{Brezis}. Thus,  in the {fourth} line we use
that $u^{i,n+1}$ is a solution for system (\ref{sys3}) where we have applied an integration by parts. In the fifth line, we used the {transposition} of the convolution (see for instance \cite[Proposition IV.16, page 67]{Brezis}), and the fact that $\check{\rho}_\eta(x)=\rho_\eta(-x)=\rho_\eta(x)$. Finally, in the last line we use that $A$ satisfies (\ref{eq+}).
\noindent Then by a straightforward recurrence we get estimate (\ref{2n}). This ends the proof of Proposition 3.2.
$\qedhere$
\end{proof}
\subsection{Passage to the limit as $(\Delta t,\e)\rightarrow (0,0)$}
In this subsection we pass to the limit as $(\Delta t,\e)\rightarrow (0,0)$ in system (\ref{sys3}) to get the existence of a solution for the continuous approximate system (\ref{sys4}) given below.\\
First, let us define the function $\Psi_{0,\ell}$ as
\begin{eqnarray}\label{psil}
\Psi_{0,\ell}(a)-\frac{1}{e}:=
\left\{\begin{array}{cll}\label{psil}
+\infty &\mbox{ if } a<0,\\\\
0 &\mbox{ if } a=0,\\\\
a\ln a &\mbox{ if } 0< a\leq\ell,\\\\
\frac{a^2}{2\ell}+a\ln\ell-\frac{\ell}{2} &\mbox{ if } a>\ell.
\end{array}\right.
\end{eqnarray}
Now let us introduce our continuous approximate system. Assume that $A$ satisfies (\ref{eq+}). Let $u_0=(u^i_0)_{1\leq i\leq m}$ satisfying 
\begin{equation}\label{C_3}
C_3:=\si\int_\O\Psi_{0,\ell}(u^i_0)<+\infty,
\end{equation}
which implies that $u^i_0\geq 0$ a.e. in $\O$ for $i=1,\ldots,m$. Then for all $\ell$, $\eta$, $\delta > 0$, with $1<\ell<+\infty$, we look for a solution $u=(u^i)_{1\leq i\leq m}$ of the following system:
\begin{eqnarray}\label{sys4}
\left\{\begin{array}{rlll}
 u^i_t&=&div\left\{{J^i_{0,\ell,\eta,\delta}(u)} \right\}&\mbox{in } \mathcal{D}'(\O_T),\\
J^i_{0,\ell,\eta,\delta}(u)&=&\displaystyle T^{0,\ell}(u^i)\left\{\sum_{j=1}^m A_{ij}\nabla\rho_{\eta}\star\rho_{\eta}\star u^j + \delta\nabla u^i\right\},&\\
u^i(0,x)&=&u^i_0(x)& \mbox{ in }\O.
\end{array}\right.
\end{eqnarray}
where $T^{0,\ell}$ is given in (\ref{T}) for $\e=0$, and we recall here $\Omega_T:=(0,T)\times\O$.
\begin{pro}\label{th3}(\textbf{Existence for system (\ref{sys4})})\\
Assume that $A$ satisfies (\ref{eq+}). Let $u_0=(u^i_0)_{1\leq i\leq m}$ satisfying (\ref{C_3}). Then for all $\ell$, $\eta$, $\delta > 0$ with $1<\ell<+\infty$ there exists a function $u=(u^i)_{1\leq i \leq m}\in (L^2(0,T;H^1(\Omega))\cap C([0,T);L^2(\O)))^m$, with $u^i\geq 0$ a.e. in $\O_T$, solution of system (\ref{sys4}) that satisfies the following entropy estimate for a.e. $t_1,t_2\in (0,T)$ with $u^i(t_1)=u^i(t_1,\cdot)$
\begin{equation}\label{2eps}
\int_{\Omega}\sum_{i=1}^m \Psi_{0,\ell}(u^i(t_2))+\delta\int_{t_1}^{t_2}\int_\Omega\sum_{i=1}^m\left|\nabla u^i\right|^2 +\delta_0\int_{t_1}^{t_2}\int_\Omega\sum_{i=1}^m\left|\nabla\rho_{\eta}\star u ^i\right|^2\ \leq\int_{\Omega}\sum_{i=1}^m\Psi_{0,\ell}(u^i_0).
\end{equation}
\end{pro}
\begin{proof}
Our proof is based on the variant of Simon's Lemma (Theorem \ref{lemsimonvar}). Recall that $\displaystyle\Delta t=\frac{T}{K}$ where $K\in\mathbb N^*$ and $T>0$ is given. We denote by $C$ a generic constant independent of $\Delta t$ and $\e$.
For all $n\in\left\{0,\ldots,K-1\right\}$ and $i=1,\ldots,m$, set
$t_n=n\Delta t$  and let the piecewise continuous function in time:
\begin{equation}\label{inter}
 U^{i,{\Delta t}}(t,x):=u^{i,n+1}(x),\quad\mbox{for } t\in (t_n,t_{n+1}],
 \end{equation}
with $U^{i,{\Delta t}}(0,x):=u^{i}_{0}(x)$ satisfying (\ref{C_1}).\\\\
\noindent {\bf Step 1: Upper bound on $\left\|\textbf{U}^{\Delta t}\right\|_{(L^2(0,T;H^1(\O)))^m}$}\\
We will prove that $U^{\Delta t}=(U^{i,\Delta t})_{1\leq i\leq m}$ satisfies
\begin{eqnarray*}
\int_{0}^{T}\|\n U^{\Delta t}(t)\|_{(L^2(\O))^m}^2 &\le &C.
\end{eqnarray*}
For all $n\in\left\{0,\ldots,K-1\right\}$ and $i=1,\ldots,m$ we have
$$\n U^{i,\Delta t}(t,x)=\n u^{i,n+1}(x), \quad\mbox{for } t\in (t_n,t_{n+1}].$$
Then
\begin{eqnarray*}
\int_{t_n}^{t_{n+1}}\|\n U^{i,\Delta t}(t)\|_{L^2(\O)}^2 &=&\Delta t\|\n u^{i,n+1}\|_{L^2(\O)}^2.
\end{eqnarray*}
Hence
\begin{eqnarray*}
\int_{0}^{T}\|\n U^{\Delta t}(t)\|_{(L^2(\O))^m}^2 &=&\Delta t\sum_{k=0}^{K-1}\|\n
u^{k+1}\|_{(L^2(\O)^m}^2 \\
&\leq& \frac{C_1}{\d},
\end{eqnarray*}
where we have used the entropy estimate (\ref{2n}) with $C_1$ is given in (\ref{C_1}).
Hence, using Poincar\'e-Wirtinger's inequality we can get similarly an upper bound on $\displaystyle\int_0^T\left\|U^{i,\Delta t}\right\|^2_{(L^2(\O))^m}$ independently of $\Delta t$ (using the fact that $\displaystyle\int_\O u^{i,n+1}=\int_\O u^{i,n}=\int_\O u^{i,0}$ by equation (\ref{sys3})) .\\\\
\noindent {\bf Step 2: Upper bound on  $\left\|\textbf{U}^{\Delta t}\right\|_{({\rm Var}([0,T);H^{-1}(\O)))^m}$}\\
We will prove that $$\left\|U^{\Delta t}\right\|_{({\rm Var}([0,T);H^{-1}(\O)))^m}\leq C. $$
We have for $i=1,\ldots,m$
\begin{eqnarray*}
\left\|U^{i,\Delta t}\right\|_{{\rm Var}([0,T);H^{-1}(\O))}&=&\sum_{n=0}^{K-1}\left\|U^{i,\Delta t}(t_{n+1})-U^{i,\Delta t}(t_n)\right\|_{H^{-1}(\O)}\\
&=&\sum_{n=0}^{K-1}\left\|u^{i,n+1}-u^{i,n}\right\|_{H^{-1}(\O)}\\
&=&\Delta t\sum_{n=0}^{K-1}\left\|\frac{u^{i,n+1}-u^{i,n}}{\Delta t}\right\|_{H^{-1}(\O)}\\
&\leq& \Delta t \sum_{n=0}^{K-1}\left\|T^{\e,\ell}(u^{i,n+1})\left(\sum_{j=1}^m A_{ij}
\nabla\rho_\eta\star\rho_\eta\star u^{j,n+1}+\d\nabla u^{i,n+1}\right)\right\|_{L^2(\O)}\\
&\leq& \ell\Delta t \sum_{n=0}^{K-1}\left\{\left\|A\right\|_{\infty}\sj\left\|\nabla\rho_\eta\star u^{j,n+1}\right\|_{L^2(\O)}+\d\left\|\nabla u^{i,n+1}\right\|_{L^2(\O)}\right\}\\
&\le& C,
\end{eqnarray*}
where 
\begin{equation}\label{norminf}
\displaystyle\left\|A\right\|_{\infty}=\max_{1\leq i\leq m}\sj \left|A_{ij}\right|,
\end{equation}
 and we have used in the last inequality the entropy estimate (\ref{2n}), and the fact that 
$$\Delta t\sum_{n=0}^{K-1}\left\|\nabla u^{i,n+1}\right\|_{L^2(\O)}\leq\sqrt{T}\left(\Delta t\sum_{n=0}^{K-1}\left\|\nabla u^{i,n+1}\right\|^2_{L^2(\O)}\right)^{\frac{1}{2}}.$$\\
\textbf{Step 3:} $\textbf{U}^{\textbf{i},\Delta t}\in \textbf{L}^{\textbf{p}}(\textbf{0},\textbf{T},\textbf{L}^\textbf{2}(\Omega))$ \textbf{with} $\textbf{p}>2$\\
The estimate (\ref{2n}) gives us that $U^{i,\Delta t} \in L^{\infty}(0,T;L^{1}(\Omega))\ \cap\ L^2(0,T;H^1(\Omega))$ for $i=1,\ldots,m$. Using Sobolev injections we get $H^1(\Omega)\hookrightarrow L^{2+\alpha(N)}(\Omega)$, with $\alpha(N)>0$, and then $U^{i,\Delta t}\in L^2(0,T;L^{2+\alpha(N)}(\Omega))$. Hence by interpolation, we find that $U^{i,\Delta t}\in L^{p}(0,T;L^{2}(\Omega))$ with $\displaystyle \left(\frac{1}{p},\frac{1}{2}\right)=(1-\theta)\left(\frac{1}{\infty},\frac{1}{2}\right)+\theta\left(\frac{1}{2},\frac{1}{2+\alpha(N)}\right)$ and $\theta\in (0,1)$, i.e. for
\begin{equation}\label{pp}
\displaystyle p=\frac{4+4\alpha(N)}{2+\alpha(N)}>2.
\end{equation}\\
\textbf{Step 4: Passage to the limit as $(\Delta \textbf{t},\e)\rightarrow (0,0)$}\\
By Steps 1,2 and 3 we have
$$\left\|U^{i,\Delta t}\right\|_{L^{p}(0,T;L^{2}(\Omega))}+\left\|U^{i,\Delta t}\right\|_{L^1(0,T;H^1(\Omega))}+\left\|{U}^{i,\Delta t}\right\|_{{\rm Var}([0,T);H^{-1}(\Omega))}\leq C.$$
Then by noticing that $H^1(\Omega)\stackrel{compact}{\hookrightarrow}L^2(\Omega)\stackrel{continous}{\hookrightarrow} H^{-1}(\Omega)$, and applying the variant of Simon's Lemma (Theorem \ref{lemsimonvar}), we deduce that $(U^{i,\Delta t})_{\Delta t}$ is relatively compact in $L^2(0,T;L^2(\Omega))$, and there exists a function $U=(U^i)_{1\leq i\leq m}\in (L^2(0,T; H^1(\O)))^m$ such that, as $(\Delta t,\e)\rightarrow (0,0)$, we have (up to a subsequence)
$$U^{i,\Delta t} \rightarrow U^i  \mbox{ strongly in } L^2(0,T; L^2(\O)).$$
By Step 1, we have $\n U^{i,\Delta t} \rightharpoonup \n U^i \mbox{ weakly in } L^2(0,T; L^2(\O))$. Now system (\ref{sys3}) can be written as 
\begin{equation}\label{jdid2}
\frac{U^{i,\Delta t}(t+\Delta t)-U^{i,\Delta t}(t)}{\Delta t}={\rm div} \left\{J^{i}_{\e,\ell,\eta,\d}(U^{i,\Delta t}(t+\Delta t),U^{i,\Delta t}(t+\Delta t))\right\} \quad\mbox{ in } \mathcal{D}'(\O_T).
\end{equation}
Multiplying this system by a test function in $\mathcal D(\O_T)$ and integrating over $\Omega_T$, we can pass directly to the limit as $(\Delta t,\e)\rightarrow (0,0)$ in (\ref{jdid2}) to get
$$U^{i}_t={\rm div} \left(T^{0,\ell}(U^{i})\left(\overset{m}{\underset{j=1}{\sum}} A_{ij}\nabla\rho_\eta\star\rho_\eta\star U^{j}+\delta\nabla U^{i}\right)\right)\quad\mbox{ in }\mathcal D'(\O_T),$$
where we used the weak $L^2$ - strong $L^2$ convergence in the products such $T^{\e,\ell}(U^{i,\Delta t})\n U^{i,\Delta t}$ to get the existence of a solution of system (\ref{sys4}). \\\\
\noindent \textbf{Step 5: Recovering the initial condition}\\
Let $\bar \rho \in C^\infty_c(\mathbb R)$ with $\bar \rho\ge 0$, $\int_\mathbb R \bar \rho =1$ and $\mbox{supp}\ \bar \rho \subset (-\frac{1}{2},\frac{1}{2})$.
We set
$$\bar \rho_{\Delta t} (t)= {\Delta t}^{-1}\bar \rho({\Delta t}^{-1}t),\mbox{ with } \bar{\rho}(t)=\bar{\rho}(-t).$$
Then we have 
\begin{eqnarray*}
&&\left\| U^{\Delta t}_t\star\bar\rho_{\Delta t}\right\|^2_{(L^2(0,T;H^{-1}(\O)))^m}=\si\int_0^T\left\|\sum_{n=0}^{K-1}(u^{i,n+1}-u^{i,n})\d_{t_{n+1}}\star \bar\rho_{\Delta t}\right\|^2_{H^{-1}(\O)}\\
&=&\si\sum_{n=0}^{K-1}\int_0^T\left(\Delta t\bar\rho_{\Delta t}(t-t_{n+1})\right)^2\left\|\frac{u^{i,n+1}-u^{i,n}}{\Delta t}\right\|^2_{H^{-1}(\O)}\\
&=&C_4\Delta t\si\sum_{n=0}^{K-1}\left\|\frac{u^{i,n+1}-u^{i,n}}{\Delta t}\right\|^2_{H^{-1}(\O)}\\
&\leq& C_4\Delta t \sum_{n=0}^{K-1}\si\int_{\O}\left|T^{\e,\ell}(u^{i,n+1})\left(\sum_{j=1}^m A_{ij}
\nabla\rho_\eta\star\rho_\eta\star u^{j,n+1}+\d\nabla u^{i,n+1}\right)\right|^2\\
&\leq&2\  C_4\ell^2\Delta t \sum_{n=0}^{K-1}\si\int_{\O}\left\{\left(\sum_{j=1}^m 
A_{ij}\nabla\rho_\eta\star\rho_\eta\star u^{j,n+1}\right)^2+\d^2\left|\nabla u^{i,n+1}\right|^2\right\}\\
&\leq&2\ C_4\ell^2\Delta t \sum_{n=0}^{K-1}\int_{\O}\left\{\left\|A\right\|^2 
\left\|\nabla\rho_\eta\star\rho_\eta\star u^{n+1}\right\|^2_{(L^2(\O))^m}+\d^2\left\|\nabla u^{n+1}\right\|^2_{(L^2(\O))^m}\right\}\\
&\le & 2\ C_4\ell^2 \Delta t \sum_{n=0}^{K-1}\left\{\|A\|^2\|\n\rho_{\eta}\star u^{n+1}\|^2_{({\lt})^m}+\d^2\left\|\n u^{n+1}\right\|^2_{(\lt)^m}\right\}\\
&\le&  2\ C_4\ell^2 C_1\left(\frac{\left\|A\right\|^2}{\d_0}+\d\right)\le  2\ C_4\ell^2 C_3\left(\frac{\left\|A\right\|^2}{\d_0}+\d\right),
\end{eqnarray*}
where $\delta_{t_{n+1}}$ is Dirac mass in $t=t_{n+1}$, $C_1$ as in (\ref{C_1}), $C_3$ as in (\ref{C_3}), $C_4:=\int_0^T\bar\rho(t)\,dt$, and we have used in the last line the entropy estimate (\ref{2n}).
Clearly, $\bar\rho_{\Delta t}\star U^{i,\Delta t}_t\rightharpoonup U^i_t$ weakly in $L^2(0,T;H^{-1}(\O))$ as $(\Delta t,\e) \rightarrow 0$.
Similarly we have that $\bar\rho_{\Delta t}\star U^{i,\Delta t}\rightarrow U^i$ strongly in $L^2(0,T;L^2(\O))$ since $U^{i,\Delta t}\rightarrow U^i$ in $L^2(0,T;L^2(\O))$. Then we deduce that $U^i\in\left\{g\in L^2(0,T;H^1(\O)); g_t\in L^2(0,T;H^{-1}(\O))\right\}$. And now $U^i(0,x)$ has sense, by Proposition \ref{compacite}, and we have that $U^i(0,x)=u^i_0(x)$ by Proposition \ref{initial}.\\\\
\noindent \textbf{Step 6: \textbf{Proof of estimate (\ref{2eps})}}\\
By Step 4, there exists a function $U^i\in{L^2(0,T;H^1(\O))}$ such that the following holds true as $(\Delta t,\e)\rightarrow (0,0)$
$$
\left\{\begin{array}{lll}
U^{i,\Delta t} &\rightarrow& U^i\\
\n U^{i,\Delta t}&\rightharpoonup &\n U^i \\ 
\n\rho_\eta\star U^{i,\Delta t}&\rightarrow &\n \rho_\eta\star U^i
\end{array}\right|\quad\mbox{in }\  L^2(0,T;L^2(\O)).
$$
Now using the fact that the norm $L^2$ is weakly lower semicontinuous, with a sequence of integers $n_2$ (depending on $\Delta t$) such that $t_{n_2+1}\rightarrow t_2\in(0,T)$ and $$U^{i,\Delta t}(t_2)=U^{i,\Delta t}(t_{n_2+1})=u^{n_2+1},$$ we get for $t_1<t_2$
\begin{equation}\label{1}
\int_{t_1}^{t_2} \int_\O \left|\n U^i\right|^2\leq\int_0^{t_2} \int_\O \left|\n U^i\right|^2\leq\displaystyle \liminf_{(\Delta t,\e)\rightarrow (0,0)} \int_0^{t_{n_2+1}} \int_\O \left|\n U^{i,\Delta t}\right|^2=\displaystyle \liminf_{(\Delta t,\e)\rightarrow (0,0)} \Delta t\sum_{k=0}^{n_2}\int_{\O}|\n u^{i,k+1}|^2,
\end{equation}
and
\begin{equation}\label{2}
\int_{t_1}^{t_2} \int_\O \left|\n\rho_{\eta}\star U^i\right|^2\leq\int_0^{t_2} \int_\O \left|\n\rho_{\eta}\star U^i\right|^2\leq\displaystyle \liminf_{(\Delta t,\e)\rightarrow (0,0)} \Delta t\sum_{k=0}^{n_2}\int_{\O}|\n\rho_\eta\star u^{i,k+1}|^2.
\end{equation}
Moreover, since we have $U^{i,\Delta t}\rightarrow U^i$ in $L^2(0,T;L^2(\O))$, we get that for a.e. $t\in (0,T)$ (up to a subsequence) $U^{i,\Delta t}(t,\cdot)\rightarrow U^i(t,\cdot)$ in $L^2(\O)$. For such $t$ we have (up to a subsequence) $U^{i,\Delta t}(t,\cdot)\rightarrow U^i(t,\cdot)$ for a.e. in $\O$. Moreover, by applying Lemma \ref{conv} we get that for a.e. $t\in (0,T)$
\begin{equation}\label{3}
\Psi_{0,\ell}(U^{i}(t))\leq \liminf_{(\Delta t,\e)\rightarrow (0,0)} \Psi_{\e,\ell}(U^{i,\Delta t}(t)).
\end{equation}
Integrating over $\O$ then applying Fatou's Lemma we get for a.e. $t_1<t_2$
\begin{equation}\label{5}
\sum_{i=1}^m\int_{\O}\Psi_{0,\ell}(U^{i}(t_2)) \leq \int_{\O}\liminf_{(\Delta t,\e)\rightarrow (0,0)}\sum_{i=1}^m\Psi_{\e,\ell}(U^{i,\Delta t}(t_2))\leq\liminf_{(\Delta t,\e)\rightarrow (0,0)}\sum_{i=1}^m\int_{\O}\Psi_{\e,\ell}(u^{i,n_2+1}).
\end{equation}
(\ref{1}),(\ref{2}) and (\ref{5}) with the entropy estimate (\ref{2n}) give us that for a.e. $t_1< t_2 \in (0,T)$
\begin{eqnarray*}
&&\displaystyle\si\int_{\O}\Psi_{\e,\ell}(U^{i}(t_2))+\d\si\int_{t_1}^{t_2} \int_\O \left|\n U^i\right|^2+\d_0\si\int_{t_1}^{t_2} \int_\O \left|\n\rho_{\eta}\star U^i\right|^2\\
 &\leq&\displaystyle\liminf_{(\Delta t,\e)\rightarrow (0,0)}\sum_{i=1}^m\int_{\O}\Psi_{\e,\ell}(u^{i,n_2+1})+\liminf_{(\Delta t,\e)\rightarrow (0,0)} \d\Delta t\si\sum_{k=0}^{n_2}\int_{\O}|\n u^{i,k+1}|^2\\
&&+\liminf_{(\Delta t,\e)\rightarrow (0,0)} \d_0\Delta t\si\sum_{k=0}^{n_2}\int_{\O}|\n \rho_\eta\star u^{i,k+1}|^2\\
&\leq&\displaystyle\sum_{i=1}^m \int_{\O} \Psi_{\e,\ell}(u^{i}_0)\leq\displaystyle\sum_{i=1}^m \int_{\O} \Psi_{0,\ell}(u^{i}_0),
\end{eqnarray*}
which is estimate (\ref{2eps}).\\\\
\noindent\textbf{Step 7: Non-negativity of $\textbf{U}^\textbf{i}$} \\
Let $\O^\e:=\left\{U^{i,\Delta t}\leq\e\right\}$. By estimate (\ref{2n}), there exists a positive
constant $C$ independent of $\e$ and $\Delta t$ such that for all $i=1,\ldots,m$ we have
\begin{eqnarray*}
C &\ge &
\int_{\O}\Psi_{\e,\ell}(U^{i,\Delta t})\\
&\geq&
\int_{\O^\e}\Psi_{\e,\ell}(U^{i,\Delta t})\\
&=&
\int_{\O^\e}\frac{1}{e}+\frac{(U^{i,\Delta t})^2}{2\e} +U^{i,\Delta t}\ln\e-\frac{\e}{2}\\
&\ge &\int_{\O^\e}\frac{1}{e}+\frac{(U^{i,\Delta t})^2}{2\e} +\e\ln\e-\frac{1}{2}\\
&\ge &\int_{\O^\e}\frac{(U^{i,\Delta t})^2}{2\e}-\frac{1}{2},
\end{eqnarray*}
i.e.
\begin{equation}\label{4}
\displaystyle \int_{\O^\e}\frac{(U^{i,\Delta t})^2}{2\e}\leq C+\frac{1}{2}.
\end{equation}
Now by passing to the limit as $(\Delta t,\e)\rightarrow (0,0)$ in (\ref{4}) we deduce that $\displaystyle\int_{\O^-}\left|U^{i}\right|^2=0$, where $\O^-:=\left\{U^{i}\leq 0\right\}$, which gives us that $(U^{i})^-=0$ in $L^2(\O)$, where $(U^{i})^-=\min(0,U^{i})$.
$\qedhere$
\end{proof}
\begin{rem}(\textbf{Another method following \cite{cont}})\\
Note that it would be also possible to use a theorem in Lions-Magenes \cite[Chap. 3, Theorem 4.1, page 257]{cont}. This would prove in particular the existence of a unique solution for the following system:
\begin{eqnarray}\label{sys2**}
\left\{\begin{array}{clll}
\displaystyle{u^i_t} &=& {\rm div} \left\{J_{\epsilon,\ell,\eta,\delta}^{i}(v,u)\right\}&\quad\mbox{ in}\quad {\mathcal D}'(\O_T),\\
J_{\epsilon,\ell,\eta,\delta}^{i}(v,u)&=&T^{\epsilon,\ell}(v^{i})\left\{\overset{m}{\underset{j=1}{\sum}} A_{ij}\nabla\rho_\eta\star\rho_\eta\star u^{j}+\delta\nabla u^{i}\right\},&\\
u^i(0,x)&=&u^i_0(x)& \quad\mbox{ in }\quad \O,
\end{array}\right.
\end{eqnarray}
 where $T^{\e,\ell}$ is given in (\ref{T}).\\It would then be possible to find a fixed point solution $v=u$ of (\ref{sys2**}) to recover a solution of (\ref{sys4}). We would have to justify again the entropy inequality (\ref{2eps}).
\end{rem}
\subsection{Passage to the limit as $(\ell,\eta)\rightarrow (\infty,0)$}
In this subsection we pass to the limit as $(\ell,\eta)\rightarrow (\infty,0)$ in system (\ref{sys4}) to get the existence of a solution for system (\ref{sys5}) given below (system independent of $\ell$ and $\eta$).\\
Let us introduce the system independant of $\ell$ and $\eta$. Asume that $A$ satisfies (\ref{eq+}). Let $u_0=(u^i_0)_{1\leq i\leq m}$ satisfying (\ref{psi0}). Then for all $\d >0$ we look for a solution $u=(u^i)_{1\leq i\leq m}$ of the following system:
\begin{eqnarray}\label{sys5}
\left\{\begin{array}{clll}
u^i_t&=&div\left\{{\displaystyle u^i\sum_{j=1}^m A_{ij}\nabla u^j}+\d u^i\n u^i \right\}&\quad\mbox{in}\quad \mathcal{D}'(\O_T),\\
u^i(0,x)&=&u_0^i(x) &\quad\mbox{a.e. in}\quad \Omega.
\end{array}\right.
\end{eqnarray}
\begin{pro}\label{th7}(\textbf{Existence for system (\ref{sys5})})\\
Assume that $A$ satisfies (\ref{eq+}). Let $u_0=(u^i_0)_{1\leq i\leq m}$ satisfying (\ref{psi0}). Then for all $\delta > 0$ there exists a function $u=(u^i)_{1\leq i \leq m}\in (L^2(0,T;H^1(\Omega))\cap C([0,T);(W^{1,\infty}(\O))'))^m$, with $u^i\geq 0$ a.e. on $\O_T$, solution of system (\ref{sys5}), that satisfies the following entropy estimate for a.e. $t_1, t_2\in (0,T)$ with $u^i(t_2)=u^i(t_2,.)$:
\begin{equation}\label{2eta}
\int_{\Omega}\sum_{i=1}^m \Psi(u^i(t_2))+\delta\int_{t_1}^{t_2}\int_{\Omega}\sum_{i=1}^m\left|\nabla u^i\right|^2 +\delta_0\int_{t_1}^{t_2}\int_{\Omega}\sum_{i=1}^m\left|\nabla u ^i\right|^2\leq\int_{\Omega}\sum_{i=1}^m\Psi(u^i_0),
\end{equation}\\
with $\Psi$ is given in (\ref{Psi}).
\end{pro}
\begin{proof}
Let $C$ be a generic constant independent of $\ell$ and $\eta$, and $u^{\ell}:=(u^{i,\ell})_{1\leq i\leq m}$ a solution of system (\ref{sys4}), where we drop the indices $\eta$ and $\delta$ to keep light notations. The proof is accomplished by passing to the limit as $(\ell,\eta)\rightarrow (\infty,0)$ in (\ref{sys4}) and using Simon's lemma (Lemma \ref{lemsimon}), in order to get the existence result.\\\\
\textbf{Step 1: Upper bound on} $\textbf{u}^{\textbf{i},\ell}_t$
\begin{eqnarray*}
&&\si\|{u}^{i,\ell}_t\|_{L^1(0,T;\left(W^{1,\infty}(\Omega))'\right)}\\
&=&\si \int^{T}_{0}\left\|div\left\{T^{0,\ell}(u^{i,\ell})\left(\sum_{j=1}^m A_{ij}\nabla\rho_{\eta}\star\rho_{\eta}\star u^{j,\ell} + \delta \nabla u^{i,\ell}\right)\right\}\right\|_{(W^{1,\infty}(\Omega))'}\\
&\leq&\si \int^{T}_{0}\int_{\O}\left|T^{0,\ell}(u^{i,\ell})\left(\sum_{j=1}^m A_{ij}\nabla\rho_\eta\star\rho_\eta\star u^{j,\ell}+\d \nabla u^{i,\ell}\right)\right|\\
&\leq& \si\int^{T}_{0}\int_{\O}\left|u^{i,\ell}\right|\left\{\left|\sum_{j=1}^m A_{ij}\nabla\rho_\eta\star\rho_\eta\star u^{j,\ell}\right|+\d \left|\nabla u^{i,\ell}\right|\right\}\\
&\leq& 	\frac{\left(1+\d\right)}{2}\left\|u^{\ell}\right\|^2_{(L^2(0,T;L^2(\O)))^m}+\frac{\left\|A\right\|^2}{2}\left\|\n \rho_\eta\star u^{\ell}\right\|^2_{(L^2(0,T;L^2(\O)))^m}+\frac{\d}{2}\left\|\n u^{\ell}\right\|^2_{(L^2(0,T;L^2(\O)))^m}\nonumber \\
&\leq& C,
\end{eqnarray*}
where we have used in the last inequality the entropy estimate (\ref{2eps}), which is also valid for $t_1=0$ and $t_2=T$, and Poincarr\'e-Wirtinger's inequality.\\\\
\textbf{Step 2: Passage to the limit as $(\ell,\eta)\rightarrow (\infty,0)$}\\
As in Step 3 of the proof of Proposition \ref{th3}, estimate (\ref{2eps}) gives us that $u^{i,\ell}\in L^{p}(0,T,L^2(\Omega))$ with $p$ is given in (\ref{pp}) and then
$$\left\|u^{i,\ell}\right\|_{L^{p}(0,T;L^{2}(\Omega))}+\left\|u^{i,\ell}\right\|_{L^1(0,T;H^1(\Omega))}+\left\|{u}^{i,\ell}_t\right\|_{L^1(0,T;(W^{1,\infty}(\Omega))')}\leq C.$$
Then by noticing that $H^1(\Omega)\stackrel{compact}{\hookrightarrow}L^2(\Omega)\stackrel{continous}{\hookrightarrow} (W^{1,\infty}(\Omega))'$, and applying Simon's Lemma (Lemma 2.3), we deduce that $(u^{i,\ell})_{\ell}$ is relatively compact in $L^2(0,T;L^2(\Omega))$, and there exists a function $u^i\in L^2(0,T; H^1(\O))$ such that, as $(\ell,\eta)\rightarrow (\infty,0)$, we have (up to a subsequence)
$$u^{i,\ell} \rightarrow u^i\mbox{ strongly in } L^2(0,T; L^2(\O)).$$
In addition, since $u^{i,\ell}\rightarrow u^i$ a.e., $u^i$ is nonnegative a.e. hence $T^{0,\ell}(u^{i,\ell}) \rightarrow u^i$ strongly in $L^2(0,T; L^2(\O))$.
Multiplying system (\ref{sys4}) by a test function in $\mathcal D(\O_T)$ and integrating over $\O_T$ we can pass
directly to the limit as $(\ell,\eta)\rightarrow (\infty,0)$, and we get 
$$u^{i}_t={\rm div} \left\{u^{i}\overset{m}{\underset{j=1}{\sum}} A_{ij}\nabla u^{j}+\d u^i\n u^i\right\}\quad\mbox{ in }\mathcal D'(\O_T).$$
where we used in particular the weak $L^2$ - strong $L^2$ convergence in the products such $T^{0,\ell}(u^{i,\ell})\n u^{i,\ell}$. Therefore, $u=(u^i)_{1\leq i\leq m}$ is a solution of system (\ref{sys5}).\\\\
\textbf{Step 3: Recovering the initial condition}\\
First of all, let $\displaystyle q=\frac{2p}{p+2}>1$, where $p>2$ is given in Step 1 of this proof.
It remains to prove that for $i=1,\ldots,m,~~\displaystyle\left\|u^{i,\ell}_t\right\|_{L^q(0,T;(W^{1,\infty})'(\O))}< C.$
We have
\begin{eqnarray*}
&&\|{u}^{i,\ell}_t\|_{L^q(0,T;\left(W^{1,\infty}(\Omega))'\right)} =
\left(\int^{T}_{0}\left\|u^{i,\ell}_t\right\|^q_{(W^{1,\infty}(\Omega))'}\right)^{\frac{1}{q}}\\
&=& \left(\int^{T}_{0}\left\|div\left\{T^{0,\ell}(u^{i,\ell})\left(\sum_{j=1}^m A_{ij}\nabla\rho_{\eta}\star\rho_{\eta}\star u^{j,\ell} + \delta \nabla u^{i,\ell}\right)\right\}\right\|^q_{(W^{1,\infty}(\Omega))'}\right)^{\frac{1}{q}}\\
&\leq& \left(\int^{T}_{0}\left(\int_{\O}\left|T^{0,\ell}(u^{i,\ell})\left(\sum_{j=1}^m A_{ij}\nabla\rho_\eta\star\rho_\eta\star u^{j,\ell}+\d \nabla u^{i,\ell}\right)\right|\right)^q\right)^{\frac{1}{q}}\\
&\leq& \left(\int^{T}_{0}\left(\int_{\O}\left|u^{i,\ell}\right|\left(\left|\sum_{j=1}^m A_{ij}\nabla\rho_\eta\star\rho_\eta\star u^{j,\ell}\right|+\d \left|\nabla u^{i,\ell}\right|\right)\right)^q\right)^{\frac{1}{q}}\\
&\leq& \left\|u^{i,\ell}\right\|_{L^p(0,T;L^2(\O))}\left\|\sj A_{ij}\n\rho_\eta\star\rho_\eta\star u^{j,\ell}\right\|_{L^2(0,T;L^2(\O))}+\d\left\|u^{i,\ell}\right\|_{L^p(0,T;L^2(\O))}\left\|\n u^{i,\ell}\right\|_{L^2(0,T;L^2(\O))}\\
&\leq& \left\|u^{i,\ell}\right\|_{L^p(0,T;L^2(\O))}\left(\left\|A\right\|_{\infty}\sj \left\| \n\rho_\eta\star u^{j,\ell}\right\|_{L^2(0,T;L^2(\O))}+\d \left\|\n u^{i,\ell}\right\|_{L^2(0,T;L^2(\O))}\right)\leq C,
\end{eqnarray*}
where we have used in the fifth line Holder's inequality (since we have $\displaystyle\frac{1}{q}=\frac{1}{p}+\frac{1}{2}$) and in the last line the entropy estimate (\ref{2eps}) and Step 1 of this proof.
Moreover, since $W^{1,1}(0,T;(W^{1,\infty}(\O))')\hookrightarrow C([0,T);(W^{1,\infty}(\O))')$ then $u^i(0,x)$ makes sense and $u^i(0,x)=u^{i}_0(x)$ for all $i=1,\ldots,m$, by Proposition \ref{initial}.\\\\
\noindent\textbf{Step 5: Proof of the estimate (\ref{2eta})}\\
The proof is similar to Step 6 of the proof of Proposition \ref{th3}.
$\qedhere$
\end{proof}
\subsection{Passage to the limit as $\d\rightarrow 0$}
\begin{proof}
Let $C$ be a generic constant independent of $\d$ and $u^{\d}:=(u^{i,\d})_{1\leq i\leq m}$ a solution of system (\ref{sys5}). We follow the lines of proof of Proposition \ref{th7}.\\
An upper bound on $u^{\textbf{i},\d}_t$ and estimate (\ref{2eta}) allow us to apply Simon's Lemma (Lemma 2.3), then $(u^{i,\d})_{\d}$ is relatively compact in $L^2(0,T;L^2(\Omega))$, and there exists a function $u^i\in L^2(0,T; H^1(\O))$ such that, as $\d\rightarrow 0$, we have (up to a subsequence)
$$u^{i,\d} \rightarrow u^i \mbox{ strongly in } L^2(0,T; L^2(\O)),$$
and
$$u^{i}_t={\rm div} \left\{u^{i}\overset{m}{\underset{j=1}{\sum}} A_{ij}\nabla u^{j}\right\}\quad\mbox{ in }\mathcal D'(\O_T).$$
Similarly to Step 4 of the proof of Proposition \ref{th7} the initial condition is recoverd. Also estimate (\ref{entro}) can be easily obtained.
\qedhere
\end{proof}
\begin{rem}\label{eed}\textbf{(Passage to the limit as $(\ell,\eta,\d)\rightarrow (\infty,0,0)$)}\\
It is possible to pass to the limit in system (\ref{sys4}) as $(\ell,\eta,\d)\rightarrow (\infty,0,0)$ at the same time: By using the entropy estimate (\ref{2eps}) and applying Simon's Lemma on the sequence $\rho_\eta\star u^{i,\ell}$ instead of $u^{i,\ell}$. Moreover, to get the entropy estimate (\ref{entro}) it is sufficient to use the fact that $\int_\O \Psi_{0,\ell} (\rho_\eta\star u^{i,\ell})\leq \int_\O \rho_\eta\star\Psi_{0,\ell}(u^{i,\ell})$.
\end{rem}
\section{Generalizations}\label{Gen}
\subsection{Generalization on the matrix A}\label{generalisation}
Assumption (\ref{eq+}) can be weaken. Indead, we can assume that $A=(A_{ij})_{1\leq i,j\leq m}$ is a real $m\times m$ matrix that satisfies a positivity condition, in the sense that there exist two positive definite diagonal $m\times m$ matrices $L$ and $R$ and $\delta_0>0$, such that we have
\begin{equation}\label{eq++}
\zeta^T L A R\, \zeta \geq \delta_0|\zeta|^2, \quad\mbox{for all}\quad \zeta \in \mathbb R^m.
\end{equation}
\begin{rem}(\textbf{Comments on the positivity condition (\ref{eq++})})\\
The assumption of positivity condition (\ref{eq++}), generalize our problem for $A$ not necessarily having a symmetric part positive definite. Here is an example of such a matrix, whose symmetric part is not definite positive, but the symmetric part of $L\,A\,R$ is definite positive for some suitable positive diagonal matrices $L$ and $R$.\\We consider
$$ A=\begin{pmatrix} 
1 & -a \\
2a & 1 
\end{pmatrix}
\mbox{with } \left|a\right|>2.$$
Indeed, $$A^{sym}=\frac{A^T+A}{2}=\begin{pmatrix} 
1 & \frac{a}{2} \\
\frac{a}{2} & 1 
\end{pmatrix},$$ satisfying $\displaystyle\det (A^{sym})=1-\frac{a^2}{4}<0$.
And let
$$L=\begin{pmatrix} 
2 & 0 \\
0 & 1 
\end{pmatrix}\quad\mbox{and}\quad
R=I_2=\begin{pmatrix} 
1 & 0 \\
0 & 1 
\end{pmatrix}.$$
On the other hand, $$B=L.A.R=\begin{pmatrix} 
2 & -2a \\
2a & 1 
\end{pmatrix},$$ satisfies that $$B^{sym}=\begin{pmatrix} 
2 & 0 \\
0 & 1 
\end{pmatrix},$$ is definite positive.
$\qedhere$
\end{rem}
\begin{pro}(\textbf{The case where $L=I_2$})\\
Let $A$ be a matrix that satisfies the positivity condition (\ref{eq++}) with $L=I_2$. Then $\bar{u}$ is a solution for system (\ref{sys0.1}) with the matrix $\bar{A}=A\,R$ (instead of $A$) if and only if $u^i=R_{ii}\,\bar{u}^i$ is a solution for system (\ref{sys0.1}) with the matrix $A$.
\end{pro}
\begin{pro}(\textbf{The case where $R=I_2$})\\
Let $u^{n+1}=(u^{i,n+1})_{1\leq i\leq m}$ be a solution of system (\ref{sys3}) with a matrix $A$ satisfying the positivity condition (\ref{eq++}) with $R=I_2$ and $L$ a positive diagonal matrix. Then $u^{n+1}$ satisfies the following entropy estimate  
\begin{eqnarray*}\label{2n**}
\sum_{i=1}^m\int_{\O}L_{ii}\Psi_{\e,\ell}(u^{i,n+1}) &+&\d\Delta t\min_{1\leq i\leq m}\left\{L_{ii}\right\}\sum_{i=1}^m\sum_{k=0}^{n}\int_{\O}|\n u^{i,k+1}|^2\\
&+&\delta_0\Delta t\sum_{i=1}^m\sum_{k=0}^{n}\int_{\O}|\n \rho_{\eta}\star u^{i,k+1}|^2\le\sum_{i=1}^m \int_{\O} L_{ii}\Psi_{\e,\ell}(u^{i}_0)
\end{eqnarray*}
\end{pro}
\begin{proof}
We have that (from fifth line of the computation in Step 5 of the proof of Proposition \ref{th2})
\begin{eqnarray*}
\sum_{i=1}^m\int_{\O}L_{ii}\left(\frac{\Psi_{\e,\ell}(u^{i,n+1})-\Psi_{\e,\ell}(u^{i,n})}{\Delta t}\right)
&\leq& -\int_{\O}\si\sj L_{ii} A_{ij}(\n\rho_\eta\star\rho_\eta\star u^{j,n+1})\cdot\n u^{i,n+1}\\
&&-\d\si\int_\O L_{ii} |\n u^{i,n+1}|^2\\
&\leq& \int_{\O}\si\sj(\n\rho_\eta\star u^{j,n+1}) L_{ii}A_{ij}(\n\rho_\eta\star u^{i,n+1})\\
&&-\d \si\int_\O L_{ii}|\n u^{i,n+1}|^2\\
&\leq& -\d_0\int_{\O}\si|\n\rho_{\eta}\star u^{i,n+1}|^2\\
&&-\d \min_{1\leq i\leq m}\left\{L_{ii}\right\} \si\int_\O |\n u^{i,n+1}|^2,
\end{eqnarray*}
where we have used, in the last line, the fact that the matrix $A$ satisfies (\ref{eq++}) with $R=I_2$. Then by a straightforward recurrence we get (\ref{2n**}).
$\qedhere$ 
\end{proof}
\begin{cor}
Theorem \ref{th0.1} still hold true if we replace condition (\ref{eq+}) by condition (\ref{eq++}).
\end{cor}
\subsection{Generalisation on the problem}
\subsubsection{{The tensor case}}
Our study can be applied on a generalized systems of the form
\begin{equation}\label{gen}
u^i_t=\sum_{j=1}^m\sum_{k=1}^N\sum_{l=1}^N\frac{\partial}{\partial x_k}\left(f_i(u^i)A_{ijkl}\frac{\partial u^j}{\partial x_l}\right)\quad\mbox{ for } i=1,\ldots,m,
\end{equation}
where $f_i$ satisfies
\begin{eqnarray*}
\left\{\begin{array}{ll}
f_i\in C(\mathbb R),&\\\\
0\leq f_i(a)\leq {C}(1+|a|) &{\quad\mbox{ for } a \in \mathbb R \mbox{ and } C>0,}\\\\
{c\left|a\right|\leq} f_i(a)&\quad\mbox{ for } a \in[0,a_0]\quad\mbox{with } a_0,c>0,\\\\
\displaystyle\int_{a_0}^A \frac{1}{f_i(a)}\,da<+\infty&\quad\mbox{ for all } A\geq a_0.
\end{array}\right.
\end{eqnarray*}
An example for such $f_i$ is
$$f_i(a)=\max\left(0,\min\left(a,\sqrt{\left|a-1\right|}\right)\right).$$
Moreover, $A=(A_{ijkl})_{i,j,k,l}$ is a tensor of order 4 that satisfies the following positivity condition: there exists $\d_0>0$ such that
\begin{equation}
\sum_{i,j,k,l} A_{ijkl}\,\eta^i\,\eta^j\,\zeta_k\,\zeta_l\,\geq\delta_0|\eta|^2|\zeta|^2 \quad\mbox{ for all } \eta\in\mathbb R^m, \zeta\in\mathbb R^N.
\end{equation}
The entropy function $\Psi_i$ is chosen such that $\Psi_i$ is nonnegative, lower semi-continuous, convex and satisfies that $\Psi_i''(a)=\displaystyle\frac{1}{f_i(a)}$ for $i=1,\ldots,m$.
Our solution satisfies the following entropy estimate for a.e. $t>0$
\begin{equation}\label{entro2}
\sum_{i=1}^m\int_{\O}{\Psi}_i(u^i(t))+\d_0\sum_{i=1}^m\int_0^{t}\int_{\O}|\nabla u^i|^2\le\sum_{i=1}^m\int_{\O}{\Psi}_i(u_0^i).
\end{equation}
To get this entropy we can apply the same strategy announced in Subsection \ref{strategy} where $f_i(u^i)$ will be replaced by $T^{\e,\ell}(f_i(v^i))$ with $T^{\e,\ell}$ given in (\ref{T}) and we use the fact that
\begin{eqnarray*}
\int_\O\sum_{i,j,k,l} \frac{\partial u^i}{\partial x_k} A_{ijkl} \frac{\partial u^{j}}{\partial x_l}&=& \sum_{n\in\mathbb{Z}^N}\sum_{i,j,k,l} \overline{\widehat{\left(\frac{\partial u^i}{\partial x_k}\right)}}(n) A_{ijkl} \widehat{\left(\frac{\partial u^{j}}{\partial x_l}\right)}(n)\\
&=&\sum_{n\in\mathbb{Z}^N}\sum_{i,j,k,l} n_k\, \widehat{u^i}(n)\, A_{ijkl}\, n_l\, \widehat{u^j}(n)\\
&\geq&\ \d_0 \sum_{n\in\mathbb{Z}^N}\left|n\right|^2 \left|\widehat{u}\right|^2=\d_0\left\|\n u\right\|_{(L^2(\O))^m}^2.
\end{eqnarray*}
\subsubsection{The variables coefficients case}
Here the coefficients $A_{ij}(x,u)$ may depend continuously of $(x,u)$. Then we have to take $\rho_\eta \star (A_{ij}(x,u) (\n\rho_\eta \star u^j))$ instead of $A_{ij} \nabla (\rho_\eta \star \rho_\eta \star u^j)$ in the approximate problem. We can consider a problem
$$ u^i_t={\rm div} \left(u^i\sj A_{ij}(x,u)\n u\right)+g^i(x,u), \quad\mbox{with } g^i\geq 0,$$
where the source terms are not too large as $u$ goes to infinity.
\subsubsection{Laplace-type equations}
Moreover, our method applies to models of the form 
\begin{equation}\label{form1}
u^i_t=\Delta (a_i(u)u^i) \quad\mbox{with } u=(u^i)_{1\leq i\leq m},
\end{equation}
under these assumptions:
\begin{eqnarray}
\left\{\begin{array}{cll}
&a_i(u)\geq 0 \quad\mbox{if}\quad u^j\geq 0 \quad\mbox{for } j=1,\ldots,m,\\
&a_i \quad\mbox{is sublinear},\\
&a_i \in C^1(\mathbb R),\\
&\displaystyle Sym\left(\left(\frac{\partial a_i}{\partial u_j}\right)_{i,j}\right)\geq \d_0 I\quad\mbox{with}\quad \d_0 >0,
\end{array}\right.
\end{eqnarray}
where $Sym$ denotes the symmetric part of a matrix.
We can consider a particular case of (\ref{form1}) where $\displaystyle a_i(u)=\sj A_{ij} u^j$. Then problem (\ref{form1}) can be written as
 \begin{equation}\label{form3}
u^i_t={\rm div} \left\{u^i\sj A_{ij} \n u^j+\left(\sj A_{ij} u_j\right) \n u^i\right\},
\end{equation}
which can be also solved under these assumptions:
\begin{eqnarray*}
\left\{\begin{array}{cll}
&A_{ij}\geq 0 \quad\mbox{for}\quad i,j=1,\ldots,m,\\
&Sym(A)\geq \d_0 I.
\end{array}\right.
\end{eqnarray*}
\subsubsection{Diffusion matrix}
We can consider the model 
\begin{equation}\label{diffusion}
u^i_t=div \left(\sum_{j=1}^m\sum_{k=1}^m B_{ijk} u^j \nabla u^k\right),\quad\mbox{ for } i=1,\ldots,m.
\end{equation}
where $Sym\left((\sum_{j=1}^m B_{ijk})_{i,k}\right)\geq \d_0I$.
\section{Appendix: Technical results}
In this section we will present some technical results that are used in our proofs. 
\begin{pro}\label{initial}(\textbf{Recovering the initial condition})\\
Let Y be a Banach space with the norm $\left\|.\right\|_Y$. Consider a sequence $(g_m)_m\in C(0,T; Y)$ such
that $(g_m)_t$ is uniformly bounded in $L^q(0,T; Y)$ with $1< q\leq \infty$, and $(g_m)_{|t=0}\rightarrow g_0$ in $Y$. Then there exists $g\in C(0,T; Y)$ such that $g_m\rightarrow g$ in $C(0,T; Y)$ and   
$$ g_{|t=0}=g_0\qquad \mbox{in } Y.$$
\end{pro}
\begin{proof}
We have that for all $s<t\in (0,T)$
\begin{eqnarray}\label{fin}
\|g_m(t)-g_m(s)\|_{Y}&=& \left\|\int_s^t(g_m)_\tau(\tau)\right\|_Y\nonumber\\
&\le & \int_s^t \|(g_m)_\tau(\tau)\|_{Y}\;ds\nonumber\\
&\le & (t-s)^{\frac{q-1}{q}}\;\|(g_m)_\tau(\tau)\|_{L^q(0,T;Y)}\leq (t-s)^{\frac{q-1}{q}} C ,
\end{eqnarray}
where we have used in the second line Holder's inequality, and the fact that $(g_m)_\tau$ is uniformly bounded in
$L^q(0,T;Y)$. Since (\ref{fin}) implies the equicontinuity of $(g_m)_m$, by Arzel\`a-Ascoli theorem, there exists $g\in C(0,T; Y)$ such that $g_m\rightarrow g$ in $C(0,T; Y)$.
Moreover, Taking $s=0$ in (\ref{fin}) we get
\begin{equation}\label{ineqdual-1}
\|g_m(t)-g_m(0)\|_{Y} \le  t^{\frac{q-1}{q}} C.
\end{equation}
By passing to the limit in $m$ in (\ref{ineqdual-1}), we deduce that
\begin{eqnarray*}
\|g(t)-g_0\|_{Y}\le t^{\frac{q-1}{q}} C
\end{eqnarray*}
Particularly, for $t=0$, we have
\begin{eqnarray*}
\|g(0)-g_0\|_{Y}=0.
\end{eqnarray*}
This implies the result.  
$\qedhere$
\end{proof}
\begin{lem}\label{conv}(\textbf{Convergence result})\\
Let $(a_\e)_\e$ a real sequence such that $a_\e \rightarrow a_0$ as $\e\rightarrow 0$. Then we have 
$$\Psi_{0,\ell}(a_0)\leq \liminf_{\e\rightarrow 0} \Psi_{\e,\ell}(a_\e),$$ where $\Psi_{\e,\ell}$ and ${\Psi_{0,\ell}}$ are given in (\ref{psi2}) and (\ref{psil}) respectively. 
\end{lem}
\begin{proof}
Consider the case where $a_0=0$.\\
We suppose that the sequence $(a_\e)_\e\in (-\infty;\frac{1}{e}]$. Let $(b_\e)_\e\in (-\infty;\frac{1}{e}]$ a sequence that decreases to $0$ as $\e\rightarrow 0$ with $b_\e >a_\e$. Since $\Psi_{\e,\ell}$ is decreasing on $(-\infty;\frac{1}{e}]$ we have $\Psi_{\e,\ell}(a_\e)\geq \Psi_{\e,\ell}(b_\e)$. Moreover, using the fact that $\Psi_{\e,\ell}(b_\e)\rightarrow 0=\Psi_{0,\ell}(0)$ we get the result.\\
Otherwise, when $(a_\e)_\e\in (\frac{1}{e};+\infty)$ the proof is the same as above but with taking $b_\e<a_\e$ since $\Psi_{\e,\ell}$ is nondecreasing in $(\frac{1}{e};+\infty)$.\\
For the other cases, $a_0<0$ and $a_0>0$, the result is easily obtained.
$\qedhere$
\end{proof}
\subsection*{Acknowledgments}
R.M. thanks his grant SAAW.

\end{document}